\documentclass[12pt]{article}

\usepackage{amsmath,amssymb}
\usepackage{verbatim}
\usepackage{color}
\newtheorem{theorem}{Theorem}[section]
\newtheorem{thm}[theorem]{Theorem}

\newtheorem{lemma}[theorem]{Lemma}
\newtheorem{proposition}[theorem]{Proposition}
\newtheorem{definition}[theorem]{Definition}

\newtheorem{remark}[theorem]{Remark}
\numberwithin{equation}{section}

\def\q{\quad}
\def\qq{{\qquad}}

\def\pd{\partial}

\def\wt{\widetilde}
\def\wh{\widehat}

\def\eps{\varepsilon}
\def\epsilon{\eps}
\def\al{\alpha}
\def\phi{\varphi}
\def\lam{{\lambda}}
\def\om{{\omega}}
\def\ol{\overline}
\def\fract{\textstyle \frac}
\def\Gam{\Gamma}\def\gam{\gamma}
\def\th{\theta}

\def\proof{{\medskip\noindent {\bf Proof. }}}

\def\qed{{\hfill $\square$ \bigskip}}

\def\dist{{\mathop {{\rm dist\, }}}}

\def\sD {{\cal D}} \def\sE {{\cal E}} \def\sF {{\cal F}}

  \def\sU {{\cal U}}

 \def\bE {{\mathbb E}}

\def\bP {{\mathbb P}}  \def\bR {{\mathbb R}}

 \def\bZ {{\mathbb Z}}

\def\nn{\nonumber}

\def\half{{\textstyle \frac12}}
\def\ignore#1{}  % use comment

\def\ms{\medskip}
\def\sms{\smallskip}

\def\sm{\smallskip\noindent}

\def\hS{\widehat S}

\def\Reff{R_{\rm eff}}
\def\wp{\wt p}

\newenvironment{shortitemize}{
\begin{enumerate}
  \setlength{\itemsep}{1pt}
  \setlength{\parskip}{0pt}
  \setlength{\parsep}{0pt}
}{\end{enumerate}}

\def\SLE{\operatorname{SLE}}
\newcommand{\Pro}[2]{\bP^{#1}\left( #2 \right)}

\newcommand{\Exp}[2]{\bE^{#1} \left[ #2 \right]}
\newcommand{\condPro}[3]{\bP^{#1} \left\{ #2 \hskip5pt \vline \hskip5pt #3 \right\}}
\newcommand{\abs}[1]{\left| #1 \right|}
\def\Lo{\operatorname{L}}
\newcommand{\ind}{\mathbf{1}}
\begin{document}

\title{\bf  Spectral dimension and random walks on the two dimensional uniform spanning tree }

\author{Martin T. Barlow\footnote{Research partially supported by NSERC (Canada) and by the
Peter Wall Institute of Advanced Studies (UBC)}  
{ }and Robert Masson\footnote{Research partially supported by NSERC (Canada)} }

\maketitle

%\ver

\begin{abstract}
  We study simple random walk on the uniform spanning tree on
  $\mathbb{Z}^2$. We obtain estimates for the transition probabilities
  of the random walk, the distance of the walk from its starting point
  after $n$ steps, and exit times of both Euclidean balls and balls in
  the intrinsic graph metric. In particular, we prove that the
  spectral dimension of the uniform spanning tree on $\mathbb{Z}^2$ is
  $16/13$ almost surely.
%\end{abstract} 



\end{abstract}
\section{Introduction} \label{sec:intro}
%\sm {\bf Acknowledgment. } 

A {\em spanning tree} on a finite graph $G=(V,E)$ is a connected subgraph 
of $G$ which is a tree and has vertex set $V$. A {\em uniform spanning tree}
in $G$ is a random spanning tree chosen uniformly from the set of all
spanning trees. Let $Q_n = [-n, n]^d \subset \bZ^d$, and write $\sU_{Q_n}$ for a
uniform spanning tree on $Q_n$.  Pemantle \cite {Pem} showed that
the weak limit of $\sU_{Q_n}$ exists and is connected if and only if $d \leq 4$. 
(He also showed that the limit does not depend on the 
particular sequence of sets $Q_n$ chosen, and that
`free' or `wired' boundary conditions give rise to the same limit.) 
We will be interested in the case $d = 2$, and will call the limit the uniform 
spanning tree (UST) on $\bZ^2$ and denote it by $\sU$.
For further information on USTs, see for example \cite{BLPS, BKPS, Lyo}. 
The UST can also be obtained as a limit as $p, q \to 0$ of the 
random cluster model -- see \cite{Hag}.

A loop erased random walk (LERW) on a graph is a process obtained by
chronologically erasing the loops of a random walk on the graph. 
%There is a strong connection between the UST and the LERW. 
There is a close connection between the UST and the LERW. 
Pemantle \cite{Pem} showed that the unique path between any two vertices $v$ and $w$ in a UST 
on a finite graph $G$ has the same distribution as the loop-erasure of a simple random walk on $G$ from $v$ to $w$. 
Wilson \cite{Wil96} then proved that a UST could be generated by a sequence of LERWs by the 
following algorithm. Pick an arbitrary vertex $v \in G$ and let $T_0 = \{v\}$. Now suppose that we have
generated the tree $T_k$ and that $T_k$ does not span. Pick any point $w \in G \setminus T_k$ 
and let $T_{k+1}$ be the union of $T_k$ and the loop-erasure of a random walk started at $w$ and
 run until it hits $T_k$. We continue this process until we generate a spanning tree
 $T_m$. Then $T_m$ has the distribution of the UST on $G$. 

We now fix our attention on $\bZ^2$. By letting the root $v$ in Wilson's algorithm go to infinity, one sees that one 
can obtain the UST $\sU$ on $\bZ^2$ by first running an infinite
LERW from a point $x_0$ (see Section \ref{sect:LERW} for the precise definition) to create the first path in $\sU$, and then 
using Wilson's algorithm to generate the rest of $\sU$. This construction makes
it clear that $\sU$ is a 1-sided tree: from each point $x$ there is a unique 
infinite (self-avoiding) path in $\sU$. 

Both the LERW and the UST on $\bZ^2$ have conformally invariant scaling limits. 
Lawler, Schramm and Werner \cite{LSW1} proved that the LERW in simply connected domains 
scales to $\SLE_2$ -- Schramm-Loewner evolution with parameter $2$. Using the relation between LERW and 
UST, this implies that the UST has a conformally invariant scaling limit in the sense of \cite{Sch} where 
the UST is regarded as a measure on the set of triples $(a,b,\gamma)$ where $a,b \in \bR^2 \cup \{\infty\}$ 
and $\gamma$ is a path between $a$ and $b$. In addition \cite{LSW1} 
proves that the UST Peano curve -- the interface between the UST and the 
dual UST -- has a conformally invariant scaling limit, which is $\SLE_8$. 

In this paper we will study properties of the UST $\sU$ on $\bZ^2$.
We have two natural metrics on $\sU$; the intrinsic metric given by the
shortest path in $\sU$ between two points, and the Euclidean metric. 
For $x,y \in \bZ^2$ let $\gam(x,y)$ be the unique path in $\sU$ between
$x$ and $y$, and let $d(x,y)=|\gam(x,y)|$ be its length. If $U_0$ is a connected
subset of $\sU$ then we write $\gam(x, U_0)$ for the unique
path from $x$ to $U_0$.
Write $\gam(x,\infty)$ for the path from $x$ to infinity.
We define balls in the intrinsic metric by
$$ B_d(x,r) =\{ y: d(x,y)\le r \}$$
and let $|B_d(x,r)|$ be the number of points in $B_d(x,r)$ (the {\em volume} of $B_d(x,r)$).
We write 
$$ B(x,r) =\{ y \in \bZ^d: |x-y| \le r\}, $$
for balls in the Euclidean metric, and let $B_R = B(R)= B(0,R)$, $B_d(R)=B_d(0,R)$. 

Our goals in this paper are to study the volume of balls in the $d$ metric, to obtain 
estimates of the degree of `metric distortion' between the intrinsic and
Euclidean metrics, and to study the behaviour of simple random walk (SRW) on $\sU$. 

To state our results we need some further notation.
Let $G(n)$ be the expected number of steps of an infinite 
LERW started at 0 until it leaves $B(0,n)$. Clearly $G(n)$ is strictly
increasing; extend $G$ to a continuous strictly increasing function from $[1,\infty)$
to $[1,\infty)$, with $G(1)=1$.
Let $g(t)$ be the inverse of $G$, so that $G(g(t))=t=g(G(t))$ for all $t\in [1,\infty)$.
By \cite{Ken, Mas09} we have
\begin{equation} \label{eq:growthexp}
 \lim_{n \to \infty} \frac{ \log G(n)} {\log n} = \frac54. 
\end{equation}

Our first result is on the relation between balls in the two metrics.

\begin{theorem} \label{t:main1}
(a) There exist constants $c, C > 0$ such that for all $r \ge 1$, $\lam \ge 1$, 
\begin{equation}\label{ei:Bdtail}
 \bP \big(  B_d(0, \lam^{-1} G(r) ) \not\subset B(0, r ) \big)  \le C e^{-c \lam^{2/3}}. 
 \end{equation} 
(b) For all $\eps > 0$, there exist $c(\epsilon), C(\epsilon) > 0$ and $\lambda_0(\eps) \geq 1$ 
such that for all $r \ge 1$ and $\lambda \geq 1$, 
\begin{equation}\label{ei:Bd-lb}
\bP \big( B(0, r)  \not\subset B_d(0, \lam G(r) \big) \leq C \lambda^{-4/15 + \eps},
\end{equation}
and for all $r \geq 1$ and all $\lambda \geq \lambda_0(\eps)$, 
\begin{equation} \label{ei:Bd-lb2} 
\bP \big( B(0, r)  \not\subset B_d(0, \lam G(r) \big) \geq c \lambda^{-4/5 - \eps}.
\end{equation} 
\end{theorem}

We do not expect any of these bounds to be optimal. In fact, we could improve the exponent 
in the bound \eqref{ei:Bdtail}, but to simplify our proofs we have not tried to find the best exponent 
that our arguments yield when we have exponential bounds. 
However, we will usually attempt to find the best exponent given by our arguments when we have polynomial bounds, 
as in \eqref{ei:Bd-lb} and \eqref{ei:Bd-lb2}. 

The reason we have a polynomial lower bound in \eqref{ei:Bd-lb2} 
is that if we have a point $w$ such that $\abs{w} = r$, then the probability that $\gamma(0,w)$ leaves 
the ball $B(0, \lambda r)$ is bounded below by $\lambda^{-1}$ (see Lemma \ref{l:Losubset}). 
This in turn implies that the probability that $w \notin B_d(0,\lam G(r))$ is bounded from 
below by $c\lambda^{-4/5-\eps}$ (Proposition \ref{p:ProMw<}).

Theorem \ref{t:main1} leads immediately to bounds on the tails
of $|B_d(0,R)|$. However,  while \eqref{ei:Bdtail} gives a good bound on the
upper tail, \eqref{ei:Bd-lb} only gives polynomial control on the
lower tail. By working harder (see Theorem \ref{t:Vexplb}) we can 
obtain the following stronger bound.

\begin{theorem} \label{t:main2}
Let $R\ge 1$, $\lam \ge 1$. Then 
\begin{align} \label{e:vlubpii}
 \Pro{}{ |B_d(0,R)| \ge \lam g(R)^2 } &\le C e^{ - c \lam^{1/3}}, \\
 \label{e:vlubpi}
 \Pro{}{ |B_d(0,R)| \le \lam^{-1} g(R)^2 } &\le C e^{-c \lam^{1/9}}. 
\end{align}
So in particular there exists $C$ such that for all $R \geq 1$,
\begin{equation}\label{e:Vdasymp2}
C^{-1} g(R)^2 \leq \Exp{}{ |B_d(0,R)| } \leq C g(R)^2. 
\end{equation}
\end{theorem}

We now discuss the simple random walk on the UST $\sU$. To help 
distinguish between the various probability laws, we will use the following notation. 
For LERW and simple random walk in $\bZ^2$ we will write 
$\bP^z$ for the law of the process started at $z$. 
The probability law of the UST will be denoted by $\bP$, and 
the UST will be defined on a probability space $(\Omega, \bP)$; we let
$\om$ denote elements of $\Omega$.
For the tree $\sU(\om)$ write $x \sim y$ if $x$ and $y$ are connected
by an edge in $\sU$, and for $x \in \bZ^2$ let
$$ \mu_x=\mu_x (\om) = | \{ y: x\sim y\} | $$
be the degree of the vertex $x$.

The random walk on $\sU(\om)$ is defined on a second space $\sD= (\bZ^2)^{\bZ_+}$. 
Let $X_n$ be the coordinate maps on
$\sD$, and for each $\om \in \Omega$ let $P^x_\om$ be the probability
on $\sD$ which makes $X=(X_n, n \ge 0)$ a simple random walk on $\sU(\om)$ started at $x$.
Thus we have $P^x_\om(X_0=x)=1$, and
$$ P^x_\om( X_{n+1}=y| X_n=x) =\frac{1}{\mu_x(\om)} \q 
\hbox{ if } y \sim x. $$
%MB changed
We remark that since the UST $\sU$ is a subgraph of $\bZ^2$ the SRW 
$X$ is recurrent. 

We define the heat kernel (transition density) with respect to $\mu$ by 
\begin{equation} \label{eq:hkdef}
p^\om_n(x,y) = \mu_y^{-1} P^x_\om(X_n =y). 
\end{equation}
Define the stopping times
\begin{align}\label{e:tRdef}
 \tau_R &= \min \{ n \ge 0: d(0, X_n) > R\},\\
 \label{e:wtrdef}
 \wt \tau_r &= \min \{ n \ge 0: |X_n| > r \}.
\end{align}
Given functions $f$ and $g$ we write
$f \approx g$ to mean
$$ \lim_{n \to \infty} \frac{ \log f(n)}{\log g(n)} = 1, $$
and $f \asymp g$ to mean that there exists $C\ge 1$ such that
$$ C^{-1} f(n) \le g(n) \le C f(n), \q n \ge 1. $$

The following summarizes our main results on the behaviour of $X$.
Some more precise estimates, including  heat kernel estimates, 
can be found in Theorems \ref{ptight} -- \ref{t:hk} in Section \ref{sect:RW}.

\begin{theorem} \label{t:mainrw}
We have for  $\bP$ -a.a. $\om$, $P^0_\om$-a.s.,
\begin{align}
\label{e:m1}
 p_{2n}(0,0) &\approx n^{-8/13}, \\
 \label{e:m2}
 \tau_R &\approx R^{13/5}, \\
 \label{e:m3}
 \wt \tau_r &\approx r^{13/4}, \\
 \label{e:m4}
\max_{0 \leq k \leq n} d(0, X_k) &\approx n^{5/13}.
  \end{align}
\end{theorem}

\sms 
We now explain why these exponents arise. If $G$ is a connected
graph, with graph metric $d$, we can define 
the volume growth exponent (called by physicists the fractal dimension of $G$) by
$$ d_f=d_f(G) = \lim_{R \to \infty} \frac { \log |B_d(0,R)|}{\log R}, $$
if this limit exists.
Using this notation, Theorem \ref{t:main2} and \eqref{eq:growthexp} imply that
$$ d_f (\sU) = 8/5, \q \bP-\hbox{a.s.} $$

Following work by mathematical physicists in the early 1980s, random walks
on graphs with fractal growth of this kind have been studied in the mathematical literature.
(Much of the initial mathematical work was done on diffusions on fractal sets, but
many of the same results carry over to the graph case). 
This work showed that the behaviour of SRW on a (sufficiently regular)
graph $G$ can be summarized by two exponents. The first of these
is the volume growth exponent $d_f$, while the second, denoted $d_w$, and called
the walk dimension, can be defined by
$$ d_w=d_w(G) = \lim_{R \to \infty} \frac { E^0 \tau_R }{\log R}
\q \hbox { (if this limit exists).}  $$
%Here $0$ is a suitable base point in the graph, and $\tau_R$ 
Here $0$ is a base point in the graph, and $\tau_R$ 
is as defined in \eqref{e:tRdef}; it is easy to see that if $G$ is connected 
then the limit is independent of the base point.
One finds that $d_f \ge 1$, $2\le d_w \le 1 + d_f$, and that all these values can
arise -- see \cite{Brev}.

Many of the early papers required quite precise knowledge of the structure
%of the graph in order to estimate $d_f$ and $d_w$. However, 
of the graph in order to calculate $d_f$ and $d_w$. However, 
\cite{BCK} showed that in some cases it is sufficient to know two facts:
the volume growth of balls, and the growth of effective resistance between points
in the graph. 
Write $\Reff(x,y)$ for the effective resistance between points $x$ and $y$ 
in a graph $G$ -- see Section \ref{sect:UST} for a precise definition.
The results of \cite{BCK} imply that if $G$ has uniformly bounded vertex degree,
and there exist $\al>0$, $\zeta>0$ such that
\begin{align}
\label{e:vgt}
     c_1 R^\al &\le \abs{B_d(x,R)} \le c_2 R^\al,  \q x \in G, \, R \ge 1,  \\
     \label{e:rgt}
    c_1 d(x,y)^\zeta &\le \Reff(x,y) \le c_2 d(x,y)^\zeta, \q x,y \in G,
\end{align}
then writing $\tau^x_R = \min\{n : d(x,X_n) > R\}$,
\begin{align} \label{e:p2n}
 p_{2n}(x,x) &\asymp n^{-\al/(\al + \zeta)}, \q x \in G, \, n \ge 1, \\
 \label{e:tdw}
 E^x \tau^x_R &\asymp R^{\al + \zeta}, \q x \in G, \, R \ge 1. 
\end{align}
(They also obtained good estimates on the transition probabilities
$P^x(X_n =y)$ -- see \cite[Theorem 1.3]{BCK}.)
 From \eqref{e:p2n} and \eqref{e:tdw} one sees that if
 $G$ satisfies \eqref{e:vgt} and \eqref{e:rgt} then
 $$ d_f = \al, \q d_w = \al + \zeta. $$

The decay $n^{-d_f/d_w}$ for the transition probabilities in \eqref{e:p2n}
can be explained as follows. If $R\ge 1$ and $2n = R^{d_w}$ then
with high probability $X_{2n}$ will be in the ball $B(x, cR)$. This ball
has $c R^{d_f} \approx c n^{d_f/d_w}$ points, and so the average
value of $p_{2n}(x,y)$ on this ball will be $n^{-d_f/d_w}$.
Given enough regularity on $G$, this average value will then
be close to the actual value of $p_{2n}(x,x)$.

In the physics literature a third exponent, called the spectral dimension,
was introduced; this can be defined by 
\begin{equation}\label{ei:ds}
  d_s(G) = -2 \lim_{n \to \infty} \frac{ \log P^x_\om( X_{2n} = x) }{\log 2n}, 
  \q \hbox { (if this limit exists).} 
\end{equation}
This gives the rate of decay of the transition probabilities; one has $d_s(\bZ^d)=d$.
The discussion above indicates that the three indices $d_f$, $d_w$ and $d_s$ 
are not independent, 
%and that in fact one expects that
and that given enough regularity in the graph $G$ one expects that
$$ d_s =\frac{ 2d_f}{d_w}. $$ 
For graphs satisfying \eqref{e:vgt} and \eqref{e:rgt} one has $d_s = 2\al/(\al + \zeta)$.

Note that if $G$ is a tree and satisfies \eqref{e:vgt} then 
$\Reff(x,y) = d(x,y)$ and so \eqref{e:rgt} holds with $\zeta=1$. Thus 
\begin{equation} \label{e:exp-tree}
 d_f = \al, \q d_w = \al +1, \q d_s = \frac{ 2 \al}{\al +1 }. 
\end{equation}

For random graphs arising from models in statistical physics, such as critical
percolation clusters or the UST, random fluctuations will mean 
that one cannot expect \eqref{e:vgt} and \eqref{e:rgt} to hold uniformly. 
Nevertheless, providing similar estimates hold with high enough
probability, it was shown in \cite{BJKS} and \cite{KM} that one can
obtain enough control on the properties of the random walk $X$ 
to calculate $d_f, d_w$ and $d_s$. 
An additional contribution of \cite{BJKS} was to show that it is sufficient to 
estimate the volume and resistance growth for balls from one base point. 
In section \ref{sect:RW}, we will use these methods to show 
that \eqref{e:exp-tree} holds for the UST, namely that

\begin{theorem} \label{t:dims}
We have for  $\bP$ -a.a. $\om$
\begin{align}\label{e:dims}
 d_f(\sU) = \frac{8}{5}, \q d_w(\sU)= \frac{13}{5}, \q d_s(\sU) = \frac{16}{13}. 
 \end{align}
\end{theorem}

The methods of \cite{BJKS} and \cite{KM} were also used in
\cite{BJKS} to study the incipient infinite cluster (IIC) for
high dimensional oriented percolation, and in \cite{KN} to show
the IIC for standard percolation in high dimensions has spectal dimension $4/3$.
These critical percolation clusters are close to trees and have 
$d_f=2$ in their graph metric. Our results for the UST are the first time
these exponents have been calculated for a two-dimensional model
arising from the random cluster model. It is natural to ask
about critical percolation in two dimensions, but in spite of
what is known via SLE, the values of $d_w$ and $d_s$ appear at
present to be out of reach.

\sms

The rest of this paper is laid out as follows. In Section
\ref{sect:LERW}, we define the LERW on $\bZ^2$ and recall the results
from \cite{Mas09, BM} which we will need.  The paper \cite{BM} gives
bounds on $M_D$, the length of the loop-erasure of a random walk run
up to the first exit of a simply connected domain $D$. However, in
addition to these bounds, we require estimates on $d(0,w)$ which by
Wilson algorithm's is the length of the loop-erasure of a random walk
started at $0$ and run up to the first time it hits $w$; we obtain
these bounds in Proposition \ref{p:ProMw<}.

In Section \ref{sect:UST}, we study the geometry of the two
dimensional UST $\sU$, and prove Theorems \ref{t:main1} and
\ref{t:main2}. In addition (see Proposition \ref{p:kmest}) we show
that with high probability the electrical resistance in the network
$\sU$ between $0$ and $B_d(0,R)^c$ is greater than $R/\lam$. The
proofs of all of these results involve constructing the UST $\sU$ in a
particular way using Wilson's algorithm and then applying the bounds
on the lengths of LERW paths from Section \ref{sect:LERW}.

In Section \ref{sect:RW}, we use the techniques from \cite{BJKS,KM}
and our results on the volume and effective resistance of $\sU$ from
%Section \ref{sect:UST} to prove Theorem \ref{t:mainrw}.
Section \ref{sect:UST} to prove Theorems \ref{t:mainrw} and \ref{t:dims}.

Throughout the paper, we use $c, c'$, $C, C'$ to denote positive
constants which may change between each appearance, but do not depend
on any variable. If we wish to fix a constant, we will denote it with
a subscript, e.g. $c_0$.

\section{Loop erased random walks} 

\label{sect:LERW}

In this section, we look at LERW on $\bZ^2$. We let $S$ be a simple random walk on $\bZ^2$, 
and given a  set $D \subset \bZ^2$, let
$$ \sigma_D = \min \{j \geq 1 : S_j \in \bZ^2 \setminus D \}$$
be the first exit time of the set $D$, and
$$ \xi_D = \min \{j \geq 1 : S_j \in D \}$$
be the first hitting time of the set $D$. If $w \in \bZ^2$, we write $\xi_w$ for $\xi_{\{ w \}}$. We also let $\sigma_R = \sigma_{B(R)}$ 
and use a similar convention for $\xi_R$.

The outer boundary of a set $D \subset \bZ^2$ is 
$$ \partial D = \{ x \in \bZ^2 \setminus D : \text{ there exists $y \in D$ such that $\abs{x-y}=1$} \},$$ 
and its inner boundary is
$$ \partial_i D = \{ x \in D : \text{ there exists $y \in \bZ^2 \setminus D$ such that $\abs{x - y}=1$} \}.$$  

Given a path $\lambda = [\lambda_0, \ldots, \lambda_m]$ in $\bZ^2$, let 
$\operatorname{L}(\lambda)$ denote its 
chronological loop-erasure. More precisely, we let
$$ s_0 = \max \{ j : \lambda(j) = \lambda(0) \},$$
and for $i > 0$,
$$ s_i = \max \{ j: \lambda(j) = \lambda(s_{i-1} + 1) \}.$$
Let $$n = \min \{ i: s_i = m \}.$$
Then 
$$ \operatorname{L} (\lambda) = [\lambda(s_0), \lambda(s_1), \ldots, \lambda(s_n)].$$

We note that by Wilson's algorithm, $\Lo(S[0,\xi_w])$ has the same distribution as $\gamma(0,w)$ -- the unique path from $0$ to $w$ in the UST $\sU$. We will therefore use $\gamma(0,w)$ to denote $\Lo(S[0,\xi_w])$ even when we make no mention of the UST $\sU$.

For positive integers $l$, let $\Omega_l$ be the set of paths
$\omega = [0, \omega_1, \ldots, \omega_k] \subset \bZ^2$ such that
$\omega_j \in B_l$, $j=1, \ldots, k-1$ and $\omega_k \in \partial B_l$. 
For $n \geq l$, define the measure $\mu_{l,n}$ on $\Omega_l$ to be the distribution on $\Omega_l$ obtained by restricting
$\Lo(S[0,\sigma_n])$ to the part of the path from $0$ to the first exit of $B_l$.

For a fixed $l$ and $\omega \in \Omega_l$, it was shown in \cite{Law91} that the sequence 
$\mu_{l,n}(\omega)$ is Cauchy. Therefore, there exists a 
limiting measure $\mu_l$ such that
$$ \lim_{n \to \infty} \mu_{l,n}(\omega) = \mu_l(\omega).$$
The $\mu_l$ are consistent and therefore there exists a measure $\mu$ on 
infinite self-avoiding paths. We call the associated process the infinite 
LERW and denote it by $\wh{S}$. We denote the exit time of a set $D$ 
for $\wh{S}$ by $\wh{\sigma}_D$. By Wilson's algorithm, $\wh{S}[0,\infty)$ has the same distribution 
as $\gamma(0,\infty)$, the unique infinite path in $\sU$ starting at $0$. 
Depending on the context, either notation will be used.

For a set $D$ containing $0$, we let $M_D$ be the number of steps 
of $\Lo(S[0,\sigma_D])$. Notice that if $D = \bZ^2 \setminus \{w\}$ and $S$ is a random walk started 
at $x$, then $M_D = d(x,w)$. In addition, if $D' \subset D$ then we let $M_{D',D}$ be 
the number of steps of $\Lo(S[0,\sigma_D])$ while it is in $D'$, or equivalently the number of points 
in $D'$ that are on the path $\Lo(S[0,\sigma_D])$.
 
We let $\wh{M}_n$ be the number of steps of $\wh{S}[0,\wh{\sigma}_n]$. As in the introduction, we set 
$G(n) = {\rm E}[\wh{M}_n]$, extend $G$ to a continuous strictly increasing function 
from $[1,\infty)$ to $[1,\infty)$ with $G(1)=1$, and let $g$ be the inverse of $G$. 
It was shown \cite{Ken, Mas09} that $G(n) \approx n^{5/4}$. In fact, the following is true.

\begin{lemma} \label{l:vgrwth}
Let $\eps>0$. Then there exist positive constants $c(\eps)$ and $C(\eps)$ such that
if $r\ge 1$ and $\lam \ge 1$, then 
\begin{align} \label{e:KMg1}
  c \lam^{5/4 -\eps} G(r)  &\le  G( \lam r)  \le  C \lam^{5/4+ \eps} G(r) , \\
  c \lam^{4/5 -\eps} g(r)  &\le  g( \lam r)  \le  C \lam^{4/5 + \eps} g(r).
\end{align}
\end{lemma}

\begin{proof} 
The first equation follows from \cite[Lemma 6.5]{BM}. Note that while the statement there 
holds only for all $r \ge R(\epsilon)$, by choosing different values of $c$ and $C$, one can easily extend it 
to all $r \ge 1$. The second statement follows from the first since $g = G^{-1}$ and $G$ is increasing. 
\qed
\end{proof}
 
The following result from \cite{BM} gives bounds on the tails of $\wh{M}_n$ and of $M_{D',D}$ 
for a broad class of sets $D$ and subsets $D' \subset D$. 
We call a subset of $\bZ^2$ {\em simply connected} if all connected components of its complement are infinite.

\begin{thm} \cite[Theorems 5.8 and 6.7]{BM} \label{t:expLERW}
There exist positive global constants $C$ and $c$, and given $\epsilon > 0$, 
there exist positive constants $C(\epsilon)$ and $c(\epsilon)$ such that for 
all $\lambda > 0$ and all $n$, the following holds.
\begin{enumerate}
\item Suppose that $D \subset \bZ^2$ contains $0$, and $D' \subset D$ is such that for all $z \in D'$, 
there exists a path in $D^c$ connecting $B(z,n+1)$ and $B(z,2n)^c$
 (in particular this will hold if $D$ is simply connected and ${\rm dist}(z,D^c) \leq n$ for all $z \in D'$). Then
\begin{equation} \label{eq:uptail}
\Pro{}{M_{D',D} > \lambda G(n)} \leq  2 e^{-c \lambda}.
\end{equation}
\item For all $D \supset B_n$,
\begin{equation} 
\Pro{}{M_D < \lambda^{-1} G(n)} \leq C(\epsilon) e^{-c(\epsilon) \lambda^{4/5 - \epsilon}}.
\end{equation}
\item  
\begin{equation} \label{eq:uptailinf}
\Pro{}{\wh{M}_n > \lambda G(n)} \leq C e^{-c \lambda}.
\end{equation}
\item 
\begin{equation} \label{eq:lowtailinf}
\Pro{}{\wh{M}_n < \lambda^{-1} G(n)} \leq C(\epsilon) e^{-c(\epsilon) \lambda^{4/5 - \epsilon}}.
\end{equation}
\end{enumerate}
\end{thm}

We would like to use \eqref{eq:uptail} in the case where $D = \bZ^2 \setminus \{w\}$ and $D' = B(0,n) \setminus \{w\}$. 
However these choices of $D$ and $D'$ do not satisfy the hypotheses in \eqref{eq:uptail}, so
we cannot use Theorem \ref{t:expLERW} directly. 
The idea behind the proof of the following proposition is to get the distribution on $\gamma(0,w)$ 
using Wilson's algorithm by first running an infinite LERW $\gamma$ (whose complement is simply connected) 
and then running a LERW from $w$ to $\gamma$.

\begin{proposition} \label{p:expbnd2} \label{p:expbnd}
There exist positive constants $C$ and $c$ such that the following holds.
Let $ n \ge 1$ and $w \in B(0,n)$. Let $Y_w = w$ if $\gamma(0,w) \subset B(0,n)$;
otherwise let $Y_w$ be the first point on the path
$\gamma(0,w)$ which lies outside $B(0,n)$. Then,
\begin{equation} \label{e:point2}
\Pro{}{ d(0,Y_w) >  \lambda G(n) } \leq C e^{-c \lambda}.
\end{equation}
\end{proposition}

\begin{proof}
Let $\gamma$ be any infinite path starting from 0, and let $\wt D = \bZ^2 \setminus \gamma$. 
Then $\wt D$ is the union of disjoint simply connected subsets
$D_i$ of $\bZ^2$; we can assume $w \in D_1$ and let $D_1 = D$.  
By \eqref{eq:uptail}, (taking $D'=B_n \cap D$) there exist $C < \infty$ and $c > 0$ such that 
\begin{equation}  \label{eq:expbnd}
\Pro{w}{M_{D',D} > \lambda G(n)} \leq C e^{-c \lambda}.
\end{equation}

Now suppose that $\gamma$ has the distribution of an infinite LERW started at 0. 
By Wilson's algorithm, if $S^w$ is an independent random walk started at $w$, 
then $\gamma(0,w)$ has the same distribution as the path from 
$0$ to $w$ in $\gamma \cup \Lo(S^w[0,\sigma_D])$. Therefore,
$$ d(0,Y_w) = \abs{\gamma(0,Y_w)} \leq \wh{M}_n + M_{D',D}, $$
and so,
\begin{align*}
 \Pro{}{d(0,Y_w) >  \lambda G(n)} 
 \leq \Pro{}{\wh{M}_n > (\lambda/2) G(n)} 
  + \max_D \Pro{w}{M_{D',D} > (\lambda/2) G(n)}.
\end{align*}
The result then follows from \eqref{eq:uptailinf} and \eqref{eq:expbnd}. \qed
\end{proof}

\begin{lemma} \label{condesc} 
There exists a positive constant $C$ such that for all $k \geq 2$, $n \geq 1$, and 
$K \subset \bZ^2 \setminus B_{4kn}$, the following holds. The probability that $\Lo(S[0,\xi_K])$ 
reenters $B_n$ after leaving $B_{kn}$ is less than $C k^{-1}$. This also holds for infinite LERWs, namely
\begin{equation}
\Pro{}{\wh{S}[\wh{\sigma}_{kn}, \infty) \cap B_{n} \neq \emptyset} \leq C k^{-1}.
\end{equation}
\end{lemma}

\begin{proof}
The result for infinite LERWs follows immediately by taking $K = \bZ^2 \setminus B_m$ and letting $m$ tend to $\infty$. 

We now prove the result for $\Lo(S[0,\xi_K])$. Let $\alpha$ be the part of the path 
$\Lo(S[0,\xi_K])$  from $0$ up to the first point $z$ where it exits $B_{kn}$. 
Then by the domain Markov property for LERW \cite{Law91}, conditioned on $\alpha$, 
the rest of $\Lo(S[0,\xi_K])$ has the same distribution as the loop-erasure of a random walk started at $z$, 
conditioned on the event $\{ \xi_K < \xi_\alpha\}$. Therefore, it is sufficient to show 
that for any path $\alpha$ from $0$ to $\partial B_{kn}$ and $z \in \partial B_{kn}$, 
\begin{eqnarray} \label{eqcondesc}
\condPro{z}{\xi_n < \xi_K}{\xi_K < \xi_\alpha} 
= \frac{\Pro{z}{\xi_{n} < \xi_K; \xi_K < \xi_\alpha}}{\Pro{z}{\xi_K < \xi_\alpha}} \leq C k^{-1}.
\end{eqnarray}

On the one hand, 
\begin{align*}
&  \Pro{z}{\xi_n < \xi_K; \xi_{K} < \xi_\alpha} \\
&\leq  \Pro{z}{\xi_{kn/2} < \xi_\alpha} \max_{x \in \partial_i B_{kn/2}} 
\Pro{x}{\xi_n < \xi_\alpha} \max_{w \in \partial B_{n}} 
\Pro{w}{\sigma_{2kn} < \xi_\alpha} \max_{y \in \partial B_{2kn}} \Pro{y}{\xi_K < \xi_\alpha}. 
\end{align*}
However, by the discrete Beurling estimates (see \cite[Theorem 6.8.1]{LL08}), 
for any $x \in \partial_i B_{kn/2}$ and $w \in \partial B_{n}$,
$$  \Pro{x}{\xi_n < \xi_\alpha} \leq C k^{-1/2};$$
$$  \Pro{w}{\sigma_{2kn} < \xi_\alpha} \leq C k^{-1/2}.$$
Therefore,
$$ \Pro{z}{\xi_n < \xi_K; \xi_{K} < \xi_\alpha} \leq C k^{-1} 
\Pro{z}{\xi_{kn/2} < \xi_\alpha} \max_{y \in \partial B_{2kn}} \Pro{y}{\xi_K < \xi_\alpha}.$$

On the other hand,
\begin{equation*}
\Pro{z}{\xi_K < \xi_\alpha} \ge 
\Pro{z}{\sigma_{2kn} < \xi_\alpha} \min_{y \in \partial B_{2kn}} \Pro{y}{\xi_K < \xi_\alpha}.
\end{equation*}
By the discrete Harnack inequality,
$$ \max_{y \in \partial B_{2kn}} \Pro{y}{\xi_K < \xi_\alpha} \leq C \min_{y \in \partial B_{2kn}} \Pro{y}{\xi_K < \xi_\alpha}.$$
Therefore, in order to prove \eqref{eqcondesc}, it suffices to show that
$$ \Pro{z}{\sigma_{2kn} < \xi_\alpha} \geq c \Pro{z}{\xi_{kn/2} < \xi_\alpha}.$$
Let $B = B(z; kn/2)$. By \cite[Proposition 3.5]{Mas09}, there exists $c > 0$ such that
$$ \condPro{z}{\abs{\arg(S(\sigma_B) - z)} \leq \pi/3}{\sigma_B < \xi_\alpha} > c.$$ 
Therefore, 
\begin{align*}
\Pro{z}{\sigma_{2kn} < \xi_\alpha} 
&\ge \sum_{\substack{y \in \partial B \\ 
 \abs{\arg(y - z)} \leq \pi/3}} \Pro{y}{\sigma_{2kn} < \xi_\alpha} \Pro{z}{\sigma_B < \xi_\alpha; S(\sigma_B) = y} \\
&\ge c \Pro{z}{\sigma_B < \xi_\alpha; \abs{\arg(S(\sigma_B) - z)} \leq \pi/3} \\
&\ge  c \Pro{z}{\sigma_B < \xi_\alpha} \\
&\ge  c \Pro{z}{\xi_{kn/2} < \xi_\alpha}.
\end{align*} \qed
\end{proof}

\begin{remark}  \label{r:shatlb}
{\rm
One can also show that there exists $\delta > 0$ such that
\begin{equation}\label{e:lhat-retn}
\Pro{}{\wh{S}[\wh{\sigma}_{kn}, \infty) \cap B_{n} \neq \emptyset} \geq c k^{-\delta}.
\end{equation}
As we will not need this bound we only give a  sketch of the proof. Since it 
will not be close to being optimal, we will not try to find the value of $\delta$ that the argument yields. 

First, we have 
\begin{align*} 
 \Pro{}{\wh{S}[\wh{\sigma}_{kn}, \infty) \cap B_{n} \neq \emptyset} 
 &\geq \Pro{}{\wh{S}[\wh{\sigma}_{kn}, \wh{\sigma}_{4kn}) \cap B_{n} \neq \emptyset}. 
\end{align*}
However, by \cite[Corollary 4.5]{Mas09}, the latter probability is comparable to the probability 
that $\Lo(S[0,\sigma_{16kn}])$ leaves $B_{kn}$ and then reenters $B_n$ before leaving $B_{4kn}$. 
Call the latter event $F$.

Partition $\bZ^2$ into the three cones $A_1 = \{z \in \bZ^2 : 0 \leq \arg(z) < 2\pi/3 \}$, $A_2 = \{z \in \bZ^2 : 2\pi/3 \leq \arg(z) < 4\pi/3 \}$ and $A_3 = \{z \in \bZ^2 : 4\pi/3 \leq \arg(z) < 2\pi \}$. Then the event $F$ contains the event that a random walk started at $0$ 
\begin{shortitemize}
\item[(1)] leaves $B_{2kn}$ before leaving $A_1 \cup B_{n/2}$,
\item[(2)] then enters $A_2$ while staying in $B_{4kn} \setminus B_{kn}$,
\item[(3)] then enters $B_n$ while staying in $A_2 \cap B_{4kn}$,
\item[(4)] then enters $A_3$ while staying in $A_2 \cap B_n \setminus B_{n/2}$, 
\item[(5)] then leaves $B_{16kn}$ while staying in $A_3 \setminus B_{n/2}$. 
\end{shortitemize}
One can bound the probabilities of the events in steps (1), (3) and (5) from below by $c k^{-\beta}$ for
some $\beta>0$. 
The other steps contribute terms that can be bounded from below by a constant; combining
these bounds gives \eqref{e:lhat-retn}.
}
\end{remark}

\begin{lemma} \label{l:Losubset} There exists a positive constant $C$ such that for all $k \geq 1$ and $w \in \bZ^2$, 
\begin{equation}
\frac{1}{8} k^{-1} \leq \Pro{}{\gamma(0,w) \not\subset B_{k\abs{w}}} \leq C k^{-1/3}.
\end{equation}
\end{lemma}

\begin{proof}
We first prove the upper bound. By adjusting the value of $C$ 
we may assume that $k \geq 4$. As in the proof of Proposition \ref{p:expbnd}, in order to 
obtain $\gamma(0,w)$, we first run an infinite 
LERW $\gamma$ started at $0$ and then run an independent random walk started at $w$ until it 
hits $\gamma$ and then erase its loops. By Wilson's algorithm, the resulting path 
from $0$ to $w$ has the same distribution as $\gamma(0,w)$.

By Lemma \ref{condesc}, the probability that $\gamma$ reenters $B_{k^{2/3}\abs{w}}$ 
after leaving $B_{k\abs{w}}$ is less than $C k^{-1/3}$. Furthermore, by the discrete Beurling estimates \cite[Proposition 6.8.1]{LL08}, 
$$ \Pro{w}{\sigma_{k^{2/3}\abs{w}} < \xi_\gamma} \leq C (k^{2/3})^{-1/2} = C k^{-1/3}.$$
Therefore, 
$$ \Pro{}{\gamma(0.w) \not\subset B_{k\abs{w}}} \leq C k^{-1/3}.$$

To prove the lower bound, we follow the method of proof of \cite[Theorem 14.3]{BLPS} 
where it was shown that if $v$ and $w$ are nearest neighbors then 
$$\Pro{}{{\rm diam} \ \gamma(v,w) \geq n} \geq \frac{1}{8n}.$$
If $w = (w_1,w_2)$, let $u = (w_1-w_2,w_1+w_2)$ and $v = (-w_2,w_1)$ so that $\{0,w,u,v\}$ form four 
vertices of a square of side length $\abs{w}$. Now consider the sets
\begin{align*}
Q_1 = \{j w : j=0,\ldots,2k\} &\quad Q_2 = \{2kw+j(u-w): j=0,\ldots,2k\} \\
Q_3 = \{j v : j=0,\ldots,2k\} &\quad Q_4 = \{2kv+j(u-v): j=0,\ldots,2k\} 
\end{align*}
and let $Q = \bigcup_{i=1}^4 Q_i$. Then $Q$ consists of $8k$ lattice points on the perimeter of a 
square of side length $2k\abs{w}$. Let $x_1, \ldots, x_{8k}$ be the ordering of these points obtained 
by letting $x_1 = 0$ and then travelling along the perimeter of the square clockwise. 
Thus $\abs{x_{i+1} - x_{i}} = \abs{w}$. Now consider any spanning tree $U$ on $\bZ^2$. 
If for all $i$, $\gamma(x_i,x_{i+1})$ stayed in the ball $B(x_i,k\abs{w})$ then the concatenation of 
these paths would be a closed loop, which contradicts the fact that $U$ is a tree. Therefore,
$$ 1 = \Pro{}{\exists i : \gamma(x_{i},x_{i+1}) \not\subset B(x_{i},k \abs{w})} 
\leq \sum_{i=1}^{8k} \Pro{}{\gamma(x_{i},x_{i+1}) \not\subset B(x_{i},k \abs{w})}.$$
Finally, using the fact that $\bZ^2$ is transitive and is invariant under rotations by $90$ degrees, all the probabilities on the right hand side are equal. This proves the lower bound. \qed

\end{proof}

\begin{proposition} \label{p:ProMw<} 
For all $\eps > 0$, there exist $c(\eps), C(\eps) > 0$ and $\lam_0(\eps) \geq 1$
such that for all $w \in \bZ^2$ and all $\lambda \geq 1$, 
\begin{equation}
\Pro{}{d(0,w) > \lambda G(\abs{w})} \leq C(\eps) \lambda^{-4/15 + \eps},
\end{equation}
and for all $w \in \bZ^2$ and all $\lambda \geq \lambda_0(\eps)$, 
\begin{equation} 
\Pro{}{d(0,w) > \lambda G(\abs{w})} \geq c(\eps) \lambda^{-4/5 - \eps}.
\end{equation} 
\end{proposition}

\begin{proof}
To prove the upper bound, let $k = \lambda^{4/5 - 3\eps}$. 
Then by Lemma \ref{l:vgrwth}, there exists $C(\eps) < \infty$ such that
\begin{equation} \label{e:ProMw1}
G(k\abs{w}) \leq C(\eps) k^{5/4 + \eps} G(\abs{w}) \leq C(\eps) \lambda^{1 - \eps} G(\abs{w}).
\end{equation}
 Then,
\begin{equation*}
\Pro{}{d(0,w) > \lambda G(\abs{w})} 
\leq \Pro{}{\gamma(0,w) \not\subset B_{k \abs{w}}} 
+ \Pro{}{d(0,w) > \lambda G(\abs{w}); \gamma(0,w) \subset B_{k \abs{w}}}.
\end{equation*}
However, by Lemma \ref{l:Losubset},
\begin{equation}
\Pro{}{\gamma(0,w) \not\subset B_{k \abs{w}}} \leq C k^{-1/3} = C \lambda^{-4/15 + \eps},
\end{equation}
while by Proposition \ref{p:expbnd} and \eqref{e:ProMw1},
\begin{align*}
\Pro{}{d(0,w) > \lambda G(\abs{w}); \gamma(0,w) \subset B_{k \abs{w}}} 
&\leq \Pro{}{d(0,w) > c(\eps) \lambda^{\eps} G(k\abs{w}); \gamma(0,w) \subset B_{k \abs{w}}} \\
&\leq C \exp(-c(\eps) \lambda^{\eps}).
\end{align*}
Therefore,
\begin{equation}
\Pro{}{d(0,w) > \lambda G(\abs{w})} \leq C \exp(-c(\eps) \lambda^{\eps}) 
+ C \lambda^{-4/15 + \eps} \leq C(\epsilon) \lambda^{-4/15 + \eps}.
\end{equation}  

\ms
To prove the lower bound we fix $k = \lambda^{4/5 + \eps}$ and assume $k \geq 2$ and $\epsilon < 1/4$. 
Then by Lemma \ref{l:vgrwth}, there exists $C(\eps) < \infty$ such that
\begin{equation*}
G((k-1)\abs{w}) \geq C(\eps)^{-1} k^{5/4 - \eps} G(\abs{w}) \geq C(\eps)^{-1} \lambda^{1 + \eps/3} G(\abs{w}).
\end{equation*}
Hence, 
$$ \Pro{}{d(0,w) > \lambda G(\abs{w})} \geq \Pro{}{d(0,w) > C(\eps) \lambda^{-\eps/3} G((k-1)\abs{w})}.$$

Now consider the UST on $\bZ^2$ and recall that $\gamma(0,\infty)$ and $\gamma(w,\infty)$ denote 
the infinite paths starting at $0$ and $w$. We write $Z_{0w}$ for the unique point where these 
meet: thus 
$\gamma(Z_{0w},\infty) = \gamma(0,\infty) \cap \gamma(w,\infty)$. 
Then $\gamma(0,w)$ is the concatenation of $\gamma(0,Z_{0w})$ and $\gamma(w, Z_{0w})$. 
By Lemma \ref{l:Losubset},
$$ \Pro{}{\gamma(0,w) \not\subset B_{k\abs{w}}} \geq \frac{1}{8k}.$$
Therefore, 
$$ \Pro{}{\gamma(0,Z_{0w}) \not\subset B_{k\abs{w}} \quad \text{or} 
\quad \gamma(w,Z_{0w}) \not\subset B_{k\abs{w}}} \geq \frac{1}{8k}.$$ 
By the transitivity of $\bZ^2$, the paths
$\gam(0, Z_{0, -w})$ and $\gam(w, Z_{0w})-w$ have the same
distribution, and therefore
$$ \Pro{}{\gamma(0,Z_{0w}) \not\subset B_{(k-1)\abs{w}}} \geq \frac{1}{16k}.$$

Since $Z_{0w}$ is on the path $\gamma(0,\infty)$, by \eqref{eq:lowtailinf},
\begin{align*}
& \Pro{}{d(0,w) > C(\eps) \lambda^{-\eps/3} G((k-1)\abs{w})} \\
& \q \geq \Pro{}{d(0,Z_{0w}) > C(\eps)\lambda^{-\eps/3} G((k-1)\abs{w})} \\
& \q \geq \Pro{}{\wh{M}_{(k-1)\abs{w}} > C(\eps) \lambda^{-\eps/3} G((k-1)\abs{w}); \gamma(0,Z_{0w}) \not\subset B_{(k-1)\abs{w}}} \\
& \q \geq \Pro{}{\gamma(0,Z_{0w}) \not\subset B_{(k-1)\abs{w}}} 
   - \Pro{}{\wh{M}_{(k-1)\abs{w}} < C(\eps) \lambda^{-\eps/3} G((k-1)\abs{w})} \\
&\q \geq \frac{1}{16k} - C \exp\{-c \lambda^{\epsilon/4}\}.
\end{align*} 
Finally, since $k = \lam^{4/5+\eps}$, the previous quantity can be made greater than 
$c(\eps) \lambda^{-4/5 - \eps}$ for $\lam$ sufficiently large. \qed
\end{proof}

%%SECTION 
\section{ Uniform spanning trees} 

\label{sect:UST}

\ms We recall that $\sU$ denotes the UST in $\bZ^2$, and we write $x \sim y$
if $x$ and $y$ are joined by an edge in $\sU$.

Let $\sE$ be the quadratic form given by
\begin{equation}
      \sE(f,g)=\fract12 \sum_{x \sim y} (f(x)-f(y))(g(x)-g(y)),
\end{equation}
If we regard $\sU$ as an
electrical network with a unit resistor on each edge, then
$\sE(f,f)$ is the energy dissipation when the vertices of $\bZ^2$ are at
a potential $f$.  Set $H^2=\{ f: \bZ^2 \to \bR:  \sE(f,f)<\infty\}$.
Let $A,B$ be disjoint subsets of $G$.  The effective resistance
between $A$ and $B$ is defined by:
\begin{equation}
\label{3.3bk}
    \Reff(A,B)^{-1}=\inf\{\sE(f,f): f\in H^2, f|_A=1, f|_B=0\}.
\end{equation}
Let $\Reff(x,y)=\Reff(\{x\},\{y\})$, and $\Reff(x,x)=0$.
For general facts on  effective resistance and its connection with
random walks see \cite{AF09,DS84,LP09}.

\bigskip

In this section, we establish the volume and effective resistance
estimates for the UST $\sU$ that will be used in the next section to
study random walks on $\sU$.

\begin{thm} \label{t:Vub}
There exist positive constants $C$ and $c$ such that for all $r \ge 1$ and $\lam > 0$,\\ 
(a)
\begin{equation}\label{e:Bdtail}
 \Pro{}{ B_d(0, \lam^{-1} G(r) ) \not\subset  B(0, r)}  \le C e^{-c \lam^{2/3}}. 
 \end{equation} 
(b)
\begin{equation} \label{e:Bdtail2}
\Pro{}{ \Reff(0, B(0,r)^c)<  \lam^{-3} G(r) } \le C e^{-c \lam^{2/3}}.
\end{equation}
\end{thm}

\begin{proof} By adjusting the constants $c$ and $C$ we can assume $\lam \ge 4$.
For $k \ge 1$, let $\delta_k = \lam^{-1} 2^{-k}$, and  $\eta_k=(2k)^{-1}$. 
Let $k_0$ be the smallest integer such that $r \delta_{k_0} < 1$. 
Set
$$ A_k = B(0, r) - B(0, (1 - \eta_k)r ), \qq k \ge 1. $$
Let $D_k$ be a finite collection of points in $A_k$ such that $|D_k|  \le   C \delta_k^{-2}$ and
\begin{align*}
A_k &\subset \bigcup_{ z \in D_k} B(z, \delta_k r).
\end{align*}

Write $\sU_1, \sU_2, \dots$ for the random trees obtained by running Wilson's algorithm
(with root $0$) with walks first starting at all points in $D_1$, then adding those points
in $D_2$, and so on. 
So $\sU_k$ is a finite tree which contains $\bigcup_{i=1}^k D_i \cup \{ 0\}$, and 
the sequence $(\sU_k)$ is increasing. 
Since $r \delta_{k_0} < 1$ we have $\pd_i B(0,r) \subset A_{k_0} \subset \sU_{k_0} $. 
We then complete a UST $\sU$ on $\bZ^2$ by applying Wilson's algorithm to the remaining points
in $\bZ^2$. 

For $z \in D_1$, let
$N_z$ be the length of the path $\gam(0,z)$ until it first exits from $B(0,r/8)$. By first applying \cite[Proposition 4.4]{Mas09} and then \eqref{eq:lowtailinf},
$$ \bP( N_z  < \lam^{-1} G(r)) \le C \bP( \wh{M}_{r/8}  < \lam^{-1} G(r)) \le C e^{-c \lam^{2/3}}, $$
so if 
$$ \wt F_1 = \{  N_z <  \lam^{-1} G(r) \hbox{ for some  $z \in D_1$ } \} 
= \bigcup_{z \in D_1} \{  N_z  < \lam^{-1} G(r) \}, $$
then 
\begin{equation}
 \bP( \wt F_1 ) \le |D_1|  C e^{-c\lam^{2/3}} \le  C  \delta_1^{-2} e^{-c \lam^{2/3}}
 \le C \lam^2 e^{-c \lam^{2/3}}. 
 \end{equation}
For $z \in A_{k+1}$, let $H_z$ be the event that the path $\gam(z,0)$ enters $B(0,(1 - \eta_k)r)$ before
it hits $\sU_k$. For $k\ge 1$, let
$$ F_{k+1} = \bigcup_{ z \in D_{k+1}} H_z. $$
Let $z \in D_{k+1}$ and $S^z$ be a simple random walk started at $z$ and run until
its hits $\sU_k$. Then by Wilson's algorithm, for the event $H_z$ to occur, $S^z$ must enter $B(0,(1 - \eta_k)r)$ before it hits $\sU_k$. 
Since each point in $A_k$ is within a distance $\delta_k r$ of $\sU_k$, 
$\sU_k$ is a connected set, and $z$ is a distance at least
$(\eta_k- \eta_{k+1})r$ from $B(0,(1 - \eta_k)r)$, we have
$$ \Pro{}{H_z}  \le \exp( -c (\eta_k- \eta_{k+1})/\delta_k ). $$
Hence
\begin{equation}
\bP( F_{k+1} ) \le C \delta_{k+1}^{-2}  \exp( -c  (\eta_k- \eta_{k+1})/\delta_k )
 \le C \lam^2 4^k \exp( -c \lam  2^k k^{-2} ).
 \end{equation}
Now define $G$ by
$$ G^c = \wt F_1 \cup \bigcup_{k=2}^{k_0}  F_k, $$
so that 
\begin{align}\label{e:PGub}
 \Pro{}{G^c} \le  C   \lam e^{-c \lam^{2/3}} + 
\sum_{k=2}^{\infty} C \lam^2 4^k \exp( -c \lam 2^k k^{-2} )
  \le C e^{-c \lam^{2/3} }.
\end{align}

Now suppose that $\omega \in G$. Then we claim that: 
\begin{shortitemize}
\item[(1)] For every $z \in D_1$ the part of the path $\gam(0,z)$ 
until its first exit from $B(0,r/2)$ is of length greater
than $\lam^{-1} G(r)$, 
\item[(2)] If $z \in D_k$ for any $k \ge 2$ then the path $\gam(z,0)$ hits $\sU_1$
before it enters $B(0,r/2)$.
\end{shortitemize}
Of these, (1) is immediate since $\omega \not\in \wt F_1$, while (2) follows by induction
on $k$ using the fact that $\omega \not\in F_k$ for any $k$.

Hence, if $\om \in G$, then $|\gam(0,z)| \ge \lam^{-1} G(r)$ for every
$z \in \pd_i B(0,r)$, which proves (a). 

\sm
To prove (b) we use the Nash-Williams bound for resistance \cite{NW}.
For $1\le k \le \lam^{-1} G(r)$ let $\Gam_k$ be the set of $z$ such that
$d(0,z)=k$ and $z$ is connected to $B(0,r)^c$ by a path in
$\{z\} \cup (\sU-\gam(0,z))$.  
Assume now that the event $G$ holds.
Then the $\Gam_k$ are disjoint
sets disconnecting $0$ and  $B(0,r)^c$, and so
$$ \Reff(0,  B(0,r)^c) \ge \sum_{k=1}^{\lam^{-1} G(r)} |\Gam_k|^{-1}. $$
Furthermore, each $z \in \Gam_k$ is on a path from $0$
to a point in $D_1$, and so $|\Gam_k| \le |D_1| \le C \delta_1^{-2} \le C \lam^2$.
Hence on $G$ we have 
$ \Reff(0, B(0,r)^c) \ge c \lam^{-3} G(r)$, which proves (b). \qed
\end{proof}

\ms
A similar argument will give a (much weaker) bound in the opposite direction. 
We begin with a result we will use to control the way
the UST fills in a region once we have constructed some initial paths. 

\begin{proposition} \label{p:fillin}
There exist positive constants $c$ and $C$ such that for each $\delta_0 \le 1$ the following holds. Let $r \geq 1$, and $U_0$ be a fixed tree in $\bZ^2$ connecting $0$ to $B(0,2r)^c$ with the property that 
${\rm dist}(x, U_0) \le \delta_0 r$
for each $x \in B(0,r)$ (here ${\rm dist}$ refers to the Euclidean distance).
Let $\sU$ be the random spanning tree in $\bZ^2$ obtained by running
Wilson's algorithm with root $U_0$. Then there exists an event
$G$ such that 
\begin{equation}\label{e:pfillub}
 \bP(G^c)  \le C e^{-c \delta_0^{-1/3}},
 \end{equation}
and on $G$ we have that for all  $x \in B(0,r/2)$,
\begin{align}\label{e:fillub}
 &d(x, U_0) \le  G( \delta_0^{1/2} r); \\
 \label{e:pathinc}
 &\gam(x, U_0) \subset B(0,r).  
 \end{align}
\end{proposition}

\begin{proof}
We follow a similar strategy to that in Theorem \ref{t:Vub}.
Define sequences $(\delta_k)$ and $(\lam_k)$ by
$\delta_k = 2^{-k} \delta_0$, $\lam_k =  2^{k/2} \lam_0$,
where $\lam_0 =   5^{-1} \delta_0^{-1/2}$.
For $k \ge 0$, let 
$$A_k = B(0, \fract12 (1+ (1+k)^{-1})r ),$$
and let $D_k \subset A_k$ be such that for $k\ge 1$,
\begin{align*}
\abs{D_k} \leq C \delta_k^{-2} , \\
A_k \subset \bigcup_{z \in D_k} B(z, \delta_k r).
\end{align*}

Let $\sU_0 = U_0$ and as before let $\sU_1, \sU_2, \ldots$ be the random trees obtained by 
performing Wilson's algorithm with root $U_0$ and starting first at points in $D_1$,
then in $D_2$ etc.  
Set
\begin{align*}
M_z &= d(z, \sU_{k-1}), \q z \in D_k, \\
F_z &= \{ \gamma(z, \mathcal{U}_{k-1}) \not\subset A_{k-1} \}, \q z \in D_k, \\
M_k &= \max_{z \in D_k} M_z, \\
F_k &= \bigcup_{z \in D_k} F_z.
\end{align*}

For $z \in D_k$, 
\begin{align} \label{e:p41a}
\Pro{}{ M_z > \lam_k G(\delta_{k-1} r) }  \leq \Pro{}{ F_z } 
+ \Pro{}{ M_z > \lam_k G(\delta_{k-1} r)  ; F_z^c } . 
\end{align}
Since $z$ is a distance at least $\fract12 r ( k^{-1} - (k+1)^{-1})$ from $A_{k-1}^c$,
and each point in $A_{k-1}$ is within a distance $\delta_{k-1} r$ of $\sU_{k-1}$,
\begin{align}  \label{e:p41b}
\Pro{}{ F_z } 
\leq C\exp(-c \delta_{k-1}^{-1} ( k^{-1} - (k+1)^{-1})) 
 \leq C\exp(-c \delta_{k-1}^{-1} k^{-2} ).
\end{align}
By \eqref{eq:uptail}, again using the fact that
each point in $A_{k-1}$ is within distance $\delta_{k-1} r$ of $\sU_{k-1}$,
\begin{align} \label{e:p41c}
\Pro{}{ M_z > \lam_k G(\delta_{k-1} r)  ; F_z^c } 
\leq C\exp(- c \lam_k^{2/3} ).
\end{align}
So, combining \eqref{e:p41a}--\eqref{e:p41c}, for $k \ge 1$, 
\begin{align} \label{e:mbnd}
\Pro{}{M_k >  \lam_k G(\delta_{k-1} r) } + \Pro{}{F_k} 
&\leq C \abs{D_k}  \left[   \exp(-c \delta_{k-1}^{-1} k^{-2} ) + \exp(- c \lam_k^{2/3} ) \right].
\end{align}

Now let 
\begin{equation}\label{e:Gdef}
 G = \bigcap_{k=1}^\infty F_k^c \cap \{ M_k \leq  \lam_k G(\delta_{k-1} r) \}.
\end{equation}
Summing the series given by \eqref{e:mbnd}, and using the bound
$|D_k| \le c \delta_k^{-2}$, we have
\begin{align*}
 \Pro{}{ G^c } 
  &\le C \delta_0^{-2}  \sum_{k}  2^{2k}
  \left[   \exp(-c \delta_{0}^{-1} 2^k k^{-2} ) + \exp(- c 2^{k/3} \delta_0^{-1/3} ) \right] \\
  &\le  C \delta_0^{-2}  e^{-c \delta_0^{-1/3}} \\
  &\le C e^{-c' \delta_0^{-1/3}}.
\end{align*}

Using Lemma \ref{l:vgrwth} with $\eps = \fract14$ gives 
$$ \lam_k G(\delta_{k-1} r ) \le \lam_k  \delta_0^{1/2} 2^{-(k-1)} G(\delta_0^{1/2} r)
= 2 \lam_0 \delta_0^{1/2} 2^{-k/2} G(\delta_0^{1/2} r). $$
So
$$ \sum_{k=1}^\infty \lam_k G(\delta_{k-1} r )  
\le   5 \lam_0 \delta_0^{1/2} G(\delta_0^{1/2} r) = G(\delta_0^{1/2} r). $$
Since $B(0,r/2) \subset \bigcap_k A_k$,
we have $B(0,r/2) \subset \bigcup_k \sU_k$. Therefore on the event $G$,
for any $x \in B(0,r/2)$, 
$d(x,U_0) \le  G(\delta_0^{1/2} r)$.
Further, on $G$, for each $z \in D_k$, we have $\gam(z, \sU_{k-1}) \subset A_{k-1}$. 
Therefore if $x \in B(0, r/2)$ the connected component of $\sU- U_0$ containing
$x$ is contained in $B(0,r)$, which proves \eqref{e:pathinc}.  \qed
\end{proof}

\begin{theorem} \label{t:Vlb}
 For all $\eps > 0$, there exist $c(\epsilon), C(\epsilon) > 0$ and $\lambda_0(\eps) \geq 1$ such that for all $r \ge 1$ and $\lambda \geq 1$, 
\begin{equation}\label{e:Bd-lb}
\bP \big( B(0, r)  \not\subset B_d(0, \lam G(r) \big) \leq C \lambda^{-4/15 + \eps},
\end{equation}
and for all $r \geq 1$ and all $\lambda \geq \lambda_0(\eps)$, 
$$ \bP \big( B(0, r)  \not\subset B_d(0, \lam G(r) \big) \geq c \lambda^{-4/5 - \eps}.$$
\end{theorem}

\begin{proof}
The lower bound follows immediately from the lower bound in Proposition \ref{p:ProMw<}.

To prove the upper bound, let $E \subset B(0,4r)$ be such that $\abs{E} \leq C \lambda^{\eps/2}$ and 
$$ B(0,4r) \subset \bigcup_{z \in E} B(z, \lambda^{-\eps/4}r).$$
We now let $\sU_0$ be the random tree obtained by applying Wilson's algorithm with 
points in $E$ and root $0$. Therefore, by Proposition \ref{p:ProMw<}, for any $z \in E$, 
\begin{align*}
\bP \big(d(0,z) > \lam G(r)/2 \big) &\leq \bP \big( d(0,z) > c \lam G(\abs{z})/2 \big)  \\
&\leq C(\eps) \lambda^{-4/15+\eps/2}.
\end{align*}
Let 
$$ F = \{ d(0,z) \le \lam G(r)/2 \, \hbox{ for all } z \in E \};$$
then
$$ \bP(F^c) \le \abs{E} C(\eps) \lam^{-4/15 + \eps/2} \leq C(\eps) \lam^{-4/15 +  \eps}. $$
We have now constructed a tree $\mathcal{U}_0$ connecting $0$ to $B(0,4r)^c$ and by the definition of the set
$E$, for all $z \in B(0,2r)$, $\dist(z,\sU_0) \leq \lam^{-\eps/4}r$. 
We now use Wilson's algorithm to produce the UST $\mathcal{U}$ on $\bZ^2$ with root $\mathcal{U}_0$. 
Let $G$ be the event given by applying Proposition 3.2 (with $r$ replaced by $2r$), so that
$$ \Pro{}{G^c} \le C e^{-c \lam^{\eps/12}}. $$
On the event $G$ we have $d(x, \sU_0) \le G(\lam^{-\eps/2}r) \le \lam G(r)/2$ for all
$ x \in B(0,r)$. Therefore, on the event 
$F \cap G$ we have  $d(x,0) \le \lam G(r)$ for all
$ x \in B(0,r)$. 
 Thus, 
$$ \Pro{}{\max_{x \in B(0,r)} d(x,0) > \lam G(r)} \leq C(\eps) \lambda^{-4/15 +  \eps} 
+ C e^{-c \lam^{\eps/12}} \leq C(\eps) \lambda^{-4/15 + \eps}.$$ 
\qed
\end{proof}

\ms
Theorem \ref{t:main1} is now immediate from Theorem \ref{t:Vub} and Theorem \ref{t:Vlb}.

\ms

While Theorem \ref{t:Vub} immediately gives the exponential bound \eqref{e:vlubpii} on the upper tail of $| B_d(0,r)|$ in Theorem \ref{t:main2}, it only gives a polynomial bound for the lower tail. The following theorem gives an exponential bound on the lower tail of $|B_d(0,r)|$ and consequently proves Theorem \ref{t:main2}.  

\begin{theorem} \label{t:Vexplb}
There exist constants $c$ and $C$ such that if $R\ge 1$, $\lam \geq 1$ then
\begin{equation}\label{e:Vexplb} 
 \bP( |B_d(0, R)| \le \lam^{-1} g(R)^2 ) \le C e^{-c \lam^{1/9}}. 
\end{equation}
\end{theorem}

\begin{proof}
Let $k \ge 1$ and let  $r = g(R/k^{1/2})$, so that $R= k^{1/2} G(r)$. 
Fix a constant $\delta_0<1$ such that the right side of
\eqref{e:pfillub} is less than $1/4$. 
Fix a further constant $\th<1$, to be chosen later but which will depend only on $\delta_0$. 

We begin the construction of $\sU$ with an infinite LERW $\hS$ 
started at 0 which gives the 
path $\gam_0=\sU_0=\gam(0, \infty)$. Let $z_i$, $i=1, \dots k$ be points on $\hS[0.\wh{\sigma}_r]$ chosen such that 
$B_i=B(z_i, r/k)$ 
are disjoint. (We choose these according to some fixed algorithm so that they depend
only on the path $\hS[0, \wh \sigma_r ]$.) 
Let
\begin{align} \label{e:F1def}
 F_1 &=\{\hbox{ $\hS[\wh \sigma_{2r}, \infty)$ hits more than $k/2$ of $B_1, \dots B_k$} \} , \\
F_2 &=\{ | \hS[0, \wh \sigma_{2r} ]| \ge \fract12 k^{1/2} G(r) \}. 
\end{align}
We have 
\begin{align} \label{e:F1}
\bP( F_1) \le C e ^{-c k^{1/3}}, \\
 \label{e:F2}
\bP( F_2) \le C e ^{-c k^{1/2}}. 
\end{align}
Of these, \eqref{e:F2} is immediate from \eqref{eq:uptailinf} 
while \eqref{e:F1} will be proved in Lemma \ref{lem:F1} below.

If either $F_1$ or $F_2$ occurs, we terminate the algorithm with a `Type 1' or `Type 2'
failure. Otherwise, we continue as follows to construct $\sU$ using
Wilson's algorithm. 

We define
$$ B_j' = B(z_i, \th r/k), \qquad B_j''= B(z_i, \th^2 r/k). $$

The algorithm is at two `levels' which we call `ball steps' and `point steps'.
We begin with a list $J_0$ of good balls. These are the balls
$B_j$ such that $B_j \cap \hS[\wh \sigma_{2r}, \infty) = \emptyset$.
The $n$th ball step starts by selecting a good ball $B_j$ from the list $J_{n-1}$ of remaining good balls. 
We then run Wilson's algorithm with paths starting in $B'_j$.
The ball step will end either with success, in which case the whole
algorithm terminates, or with one of three kinds of failure.
In the event of failure the ball $B_j$, and possibly
a number of other balls also, will be labelled `bad', and
$J_{n}$ is defined to be the remaining set of good balls.
If more than $k^{1/2}/4$ balls are labelled bad at any one ball step, we terminate the
whole algorithm with a `Type 3 failure'. Otherwise, we proceed until, if 
we have tried $k^{1/2}$ balls steps
without a success, we terminate the algorithm with a `Type 4 failure'.

We write $\sU_n$ for the tree obtained after $n$ ball steps. 
After ball step $n$, any ball $B_j$ in $J_{n}$ will have the property that
$B_j' \cap \sU_n = B'_j \cap \sU_0$.

We now describe in detail the second level of the algorithm, which works with a 
fixed (initially good) ball $B_j$. We assume that this is the $n$th ball
step (where $n \ge 1$), so that 
we have already built the tree $\sU_{n-1}$.
Let $D' \subset B(0, \th^2 r/k)$ satisfy
$$ |D'| \le c \delta_0^{-2}, \qquad 
B(0, \th^2 r/k) \subset \bigcup_{x \in D'} B(x, \delta_0 \th^2 r/k ). $$
Let $D_j = z_j + D'$, so that $D_j \subset B''_j$.

We now proceed to use Wilson's algorithm to build the paths
$\gam(w, \sU_{n-1})$ for $w \in D_j$. For $w \in D_j$ let
$S^w$ be a random walk started at $w$.
For each $w \in D_j$ let $G_{w}$ be the event that
$\gam(w, \sU_{n-1}) \subset B_j'$. 
If $F_{w}$ is the event that $S^{w}$ exits from $B'_j$ before it hits $\sU_0$, then 
\begin{equation}\label{e:Gwc}
 \bP( G^c_w) \le \bP(F_{w}) \le c \theta^{1/2}.
\end{equation}
Here the first inequality follows from Wilson's algorithm,
while the second is by the discrete Beurling estimates (\cite[Proposition 6.8.1]{LL08}).

Let $M_{w}= d(w, \sU_{n-1})$,
and $T_{w}$ be the first time $S^{w}$  hits $\sU_{n-1}$.
 Then by  Wilson's algorithm and \eqref{eq:uptail},
\begin{align}   \label{e:MwGw}
  \bP( M_{w} \ge \th^{-1}  G(\th r/k); G_w )
= \bP( M_{w} \ge \th^{-1} G(\th r/k); \Lo(S^{w}[0, T_{w}]) \subset  B'_j )  
 \le   c e^{ - c \th^{-1}}. 
\end{align}

We now define sets corresponding to three possible outcomes to this procedure:
\begin{align*}
  H_{1,n} &=  \bigcup_{w \in D_j} G^c_w, \\
 H_{2,n} &=  \left\{ \max_{w \in D_j} M_w \ge \th^{-1} G(\th r/k)\right\} \cap \bigcap_{w \in D_j} G_w, \\
 H_{3,n} &=  \left\{ \max_{w \in D_j} M_w < \th^{-1} G(\th r/k)\right\} \cap \bigcap_{w \in D_j} G_w.
\end{align*}
By \eqref{e:Gwc},
\begin{equation}\label{e:H1ub}
 \bP(H_{1,n}) \le \sum_{w \in D_j}  \bP(G_w) \le  c \delta_0^{-2} \th^{1/2},
\end{equation}
and by \eqref{e:MwGw},
\begin{equation}\label{e:H2ub}
 \bP(H_{2,n}) \le \sum_{w \in D_j} \bP( M_{w} \ge \th^{-1} G(\th r/k); G_w)
 \le c \delta_0^{-2} e^{ - c \th^{-1}} .
\end{equation}
We now choose the constant $\th$ small enough so that each of
$\bP(H_{i,n}) \le \fract14$ for $i=1,2$, and therefore  
\begin{equation}\label{e:H3lb}
 \bP(H_{3,n}) \ge \fract12. 
\end{equation}

If $H_{3,n}$ occurs then we have constructed a tree $\sU_n'$ which contains
$\sU_{n-1}$ and $D_j$. Further, we have that for each point $w \in D_j$,
the path $\gam(w,0)$ hits $\sU_0$ before it leaves $B'_j$. Hence,
$$ d(w,0) \le M_w + \max_{z \in \sU_0 \cap B_j} d(0,z)
 \le \fract12 k^{1/2} G(r) + \th^{-1} G(\th r/k). $$
 
 We now use Wilson's algorithm to fill in the remainder of $B'_j$.
 Let $G_n$ be the event given by applying Proposition \ref{p:fillin}
 to the ball $B''_j$ with $U_0= \sU'_{n}$. 
 Then
$$ \bP( G_{n}^c ) \le c e^{-c \delta_0^{-1/3} } \le \fract14 $$
by the choice of $\delta_0$, 
and therefore $\bP(H_{3,n} \cap G_n) \ge \fract14$.
If this event occurs, then all points in $B(z_j, \th^2 r/2k)$ are within distance
 $G( \delta_0^{1/2} \th^2 r/k)$ of $\sU'_n$ in the graph metric $d$; in this case 
 we label ball step $n$
as successful, and we terminate the whole algorithm. 
Then for all $z \in B(z_j, \th^2 r/2k)$,
\begin{align*} 
d(0,z) &\le d(z, \sU_n') + \max_{w \in \sU_n'} d(w,0) \\
&\le G( \delta_0^{1/2} \th^2 r/k) +  \fract12 k^{1/2} G(r) + \th^{-1} G(\th r/k)\\
&\le  k^{1/2} G(r),
 \end{align*}
provided that $k$ is large enough.
So there exists $k_0\ge 1$ such that, provided that $k\ge k_0$,
if $H_{3,n} \cap G_n$ occurs then 
$B(z_j, \th^2 r/2k ) \subset B_d (0, k^{1/2} G(r) )$. 
Since $R = k^{1/2} G(r) \le G(k^{1/2} r)$ we have
$g(R) \le k^{1/2} r$, and therefore 
\begin{equation}\label{e:succ}
 |B_d(0, R)| \ge  | B(z_j, \th^2 r/2k)| \ge c k^{-2} r^2
  \ge c g(R)^2/k^3. 
  \end{equation}

\sms
If $H_{1,n} \cup H_{2,n} \cup (H_{3,n} \cap G_n^c)$ occurs 
then as soon as we have a random walk $S^w$ that 
`misbehaves' (either by leaving $B_j'$ before hitting $\sU_0$, or by having
$M_w$ too large), then we terminate the ball step and mark the ball $B_j$ as `bad'.
If $\om \in H_{2,n}$ only the ball $B_j$ becomes bad, but
if $\om \in H_{1,n} \cup (H_{3,n}\cap G_n^c) $  
 then $S^w$ may hit several other balls $B'_i$ before it 
hits $\sU_{n-1}$. Let $N^B_w$ denote the number of such balls hit by $S^w$.
By Beurling's estimate, the probability that
$S^w$ enters a ball $B_i'$ and then exits $B_i$ without hitting $\sU_0$ is less
than $c \th^{1/2}$. Since the balls $B_i$ are disjoint, 
\begin{equation} \label{e:NBtail}
 \bP( N^B_w \ge m ) \le (c \th^{1/2})^m \le e^{-c' m}. 
\end{equation}
A Type 3 failure occurs if $N^B_w\ge k^{1/2}/4$; using \eqref{e:NBtail} we see
that the probability that a ball step ends with a Type 3 failure is bounded by
$\exp(- c k^{1/2})$. If we write $F_3$ for the event that some ball
step ends with a Type 3 failure, then since there are at most
$k^{1/2}$ ball steps,
\begin{equation}\label{e:F3}
 \bP(F_3) \le k^{1/2} \exp(- c k^{1/2}) \le  C \exp(- c' k^{1/2}).
\end{equation}

The final possibility is that $k^{1/2}$ ball steps all end in failure;
write $F_4$ for this event. Since each ball step has a probability
at least $1/4$ of success (conditional on the previous steps 
of the algorithm), we have
\begin{equation}\label{e:F4}
 \bP(F_4) \le (3/4)^{k^{1/2}} \le  e^{-c k^{1/2}}. 
\end{equation}

Thus either the algorithm is successful, or it ends with one of four
types of failure, corresponding to the events $F_i$, $i=1, \dots 4$.
By Lemma \ref{lem:F1} and 
\eqref{e:F2}, \eqref{e:F3}, \eqref{e:F4} we have
$\bP(F_i) \le C \exp(-c k^{1/3})$ for each $i$. 
Therefore, we have that provided $k \ge k_0$, \eqref{e:succ} holds
except on an event of probability $C \exp(-c k^{1/3})$.
Taking $k=c \lam^{1/3}$ for a suitable constant $c$, and adjusting
the constant  $C$ so that \eqref{e:Vexplb} holds for all $\lam$ completes
the proof. \qed
\end{proof}

\ms

The reason why we can only get a polynomial bound in the Theorem \ref{t:Vlb} is that one cannot 
get exponential estimates for the probability that $\gamma(0,w)$ leaves $B(0,k\abs{w})$ 
(see Lemma \ref{l:Losubset}). However, if we let $U_r$ be the connected component of 
$0$ in $\sU \cap B(0,r)$, then the following proposition enables us to get exponential control on the length 
of $\gamma(0,w)$ for $w \in U_r$. This will allow us to obtain an exponential bound on the lower 
tail of $\Reff(0, B_d(0,R)^c)$ in Proposition \ref{p:kmest}. 

\begin{proposition} \label{p:Ur}
There exist positive constants $c$ and $C$ such that for all $\lam \ge 1$ and $r \ge 1$, 
\begin{equation}\label{e:Uri}
\Pro{}{ U_r \not\subset B_d(0,\lam G(r)) } \le  C e^{-c \lam}. 
\end{equation}
\end{proposition}

\begin{proof}
This proof is similar to that of Theorem \ref{t:Vlb}. 
Let $E \subset B(0,2r)$ be such that
$\abs{E} \leq C \lambda^6$ and 
$$ B(0,2r) \subset \bigcup_{z \in E} B(z, \lambda^{-3} r),$$
and let $\sU_0$ be the random tree obtained by applying Wilson's algorithm with 
points in $E$ and root $0$. 
For each $z \in E$, let $Y_z$ be defined as in Proposition \ref{p:expbnd2},
so that $Y_z = z$ if $\gam(0,z) \subset B(0,2r)$, and 
otherwise $Y_z$ is the first point on $\gam(0,z)$ which is outside $B(0,2r)$.
Let
\begin{align*}
G_1 &= \{ d(Y_z,0) \le \fract12 \lam G(r) \, \hbox{ for all } z \in E \}.   
\end{align*}
Then by Proposition \ref{p:expbnd2}, 
\begin{equation}\label{e:pg1}
\bP( G_1^c) \le \sum_{z \in E} \bP( d(Y_z,0) > \half \lam G(2 r) ) 
\le \abs{E} C e^{-c \lam}  \le C \lam^6 e^{-c\lam}. 
\end{equation}
We now complete the construction of $\sU$ by using
Wilson's algorithm. Then Proposition \ref{p:fillin} with $\delta_0 = \lam^{-3}$ implies 
that there exists an event $G_2$ with
\begin{equation}\label{e:pg2}
\Pro{}{G_2^c} \le  e^{-c \delta_0^{-1/3}} = e^{-c \lam},
\end{equation}
and on $G_2$,
\begin{align*}
 &\max_{x \in B(0,r)} d(x, \sU_0) \le  G( \lam^{-3/2} r).  
 \end{align*}

Suppose $G_1 \cap G_2$ occurs, and let $x \in U_r$. Write $Z_x$
for the point where $\gam(x,0)$ meets $\sU_0$. Since $x \in U_r$,
we must have $Z_x \in B(0,r)$, and $\gam(Z_x, 0) \subset B(0,r)$.
As $Z_x \in \sU_0$, there exists $z \in E$ such that $Z_x \in \gam(0,z)$.
Since $G_1$ occurs, $d(0, Z_x) \le d(0, Y_z) \le \half \lam G(r)$,
while since $G_2$ occurs $d(x,Z_x) \le G(\lam^{-3/2}r)$. So, provided
$\lam$ is large enough,
$$ d(0,x) \le d(0, Z_x) + d(Z_x,x) \le \half \lam G(r)+ G(\lam^{-3/2}r)
\le \lam G(r). $$
Using \eqref{e:pg1} and \eqref{e:pg2}, and adjusting the constant
$C$ to handle the case of small $\lam$ completes the proof. \qed
\end{proof}

\begin{proposition} \label{p:kmest}
There exist positive constants $c$ and $C$ such that for all $R\ge 1$ and $\lam \ge 1$, \\
(a)
\begin{equation}\label{e:Reffa}
\bP( \Reff(0, B_d(0,R)^c) <  \lam^{-1} R ) \le C e^{-c \lam^{2/11}};
\end{equation}
(b) 
\begin{equation}
 \bE ( \Reff(0, B_d(0,R)^c) |B_d(0,R)|) \le C R g(R)^2.
\end{equation}
\end{proposition}

\begin{proof}
(a) Recall the definition of $U_r$ given before Proposition \ref{p:Ur}, 
and note that for all $r \geq 1$, $\Reff(0, B(0,r)^c) = \Reff(0, U_r^c)$.
Given $R$ and $\lam$, let $r$ be such that $R= \lam^{2/11} G(r)$.
By monotonicity of resistance we have that if
$U_r \subset B_d(0,R)$, then 
$$ \Reff(0, B_d(0,R)^c) \ge \Reff(0, U_r^c). $$
So, writing $B_d = B_d(0,R)$,
\begin{align*}
\bP( \Reff(0, B_d^c) <  \lam^{-1} R )
 &= \bP( \Reff(0, B_d^c) < \lam^{-1} R; U_r \not\subset B_d) 
+ \bP(  \Reff(0, B_d^c) < \lam^{-1} R; U_r \subset B_d)\\ 
 &\le \bP( U_r \not\subset B_d(0, \lam^{2/11} G(r) ) )
+ \bP(  \Reff(0, U_r^c) < \lam^{-9/11} G(r) ).
\end{align*}
By Proposition \ref{p:Ur},
$$ \bP( U_r \not\subset B_d(0, \lam^{2/11} G(r) ) ) \leq C e^{-c\lam^{2/11}},$$
while by \eqref{e:Bdtail2},
$$ \bP(  \Reff(0, U_r^c) < \lam^{-9/11} G(r) ) \leq C e^{-c \lam^{2/11}}.$$
This proves (a).

(b) Since $\Reff(0, B_d(0,R)^c) \le R$, this is immediate 
from Theorem \ref{t:main2}. \qed
\end{proof}

\ms

We conclude this section by proving the following technical lemma that
was used in the proof of Theorem \ref{t:Vexplb}.

\begin{lemma}\label{lem:F1}
Let $F_1$ be the event defined by \eqref{e:F1def}. Then
\begin{align}
\bP( F_1) \le C e ^{-c k^{1/3} }. 
\end{align}
\end{lemma}
 
\begin{proof}
Let $b= e^{k^{1/3}}$. Then by Lemma \ref{condesc}
\begin{equation}\label{e:expret}
\Pro{}{\wh{S}[\wh{\sigma}_{br}, \infty) \cap B_{r} \neq \emptyset} \le C b^{-1}
\le C e^{-k^{1/3}}. 
\end{equation}
If $\hS[\wh \sigma_{2r}, \infty)$ hits more than $k/2$ balls 
then either $\hS$ hits $B_r$ after time $\wh \sigma_{br}$, or
$\hS[\wh \sigma_{2r}, \wh \sigma_{br}]$ hits more than $k/2$ balls.
Given \eqref{e:expret}, it is therefore sufficient to prove that
\begin{equation}\label{e:43b}
 \bP( \hS[\wh \sigma_{2r}, \wh \sigma_{br}] \hbox{ hits more than $k/2$ balls} )
 \le  C e ^{-c k^{1/3} }. 
\end{equation}
Let $S$ be a simple random walk started at $0$, and let 
$L' = \Lo(S[0, \sigma_{4br}])$. Then  by \cite[Corollary 4.5]{Mas09}, in order to prove
\eqref{e:43b}, it is sufficient to prove that
\begin{equation}\label{e:43c}
 \bP( L' \hbox{ hits more than $k/2$ balls} )
 \le  C e ^{-c k^{1/3} }. 
\end{equation}

Define stopping times for $S$ by letting $T_0=\sigma_{2r}$ and for $j \ge 1$,
\begin{align*}
 R_j &= \min\{ n \ge T_{j-1}: S_n \in B(0, r) \}, \\
 T_j &= \min\{ n \ge R_{j}: S_n \notin B(0, 2r) \}.
\end{align*}
Note that the balls $B_j$ can only be hit by $S$ in the intervals
$[R_j, T_j]$ for $j \ge 1$.
Let $M =\min\{j: R_j \ge \sigma_{4 br} \}$. Then
$$ \bP( M = j+1 | M>j) =  \frac{ \log (2r) -\log (r)}{\log ( 4br) - \log r} 
= \frac{\log 2}{ \log (4b)} \ge c k^{-1/3}. $$
Hence
$$ \bP( M \ge k^{2/3}) \le C \exp( - c k^{1/3} ). $$

For each $j\ge 1$ let $L_j = \Lo(S[0, T_j])$, 
let $\al_j$ be the first exit by $L_j$ from $B(0,2r)$, and
$\beta_j$ be the number of steps of $L_j$. 

If $L'$ hits more than $k/2$ balls
then there must exist some $j \le M$ such that
$L_j[\al_j, \beta_j]$ hits more than $k/2$ balls $B_i$.
(We remark that since the balls $B_i$ are defined in terms of the
loop erased walk path, they will depend on $L_j[0, \al_j]$.
However, they will be fixed in each of the
intervals $[R_j, T_j]$.)
Hence, if $M \le k^{2/3}$ and $L'$ hits more than $k/2$ balls
then $S$ must hit more than $c k^{1/3}$ balls in one of the
intervals $[R_j, T_j]$, without hitting the path $L_j[0, \al_j]$.
However, by Beurling's estimate
the probability of this event is less than 
$C \exp(-c k^{1/3})$. Combining these estimates 
concludes the proof. \qed
\end{proof}

%%SECTION 

\section{Random walk estimates}\label{sect:RW}
We recall the notation of random walks on the UST given in the introduction.
In addition, define $P^*$ on $\Omega \times \sD$ by
setting $P^*( A \times B) = \bE [\ind_A P^0_\om (B)]$ and extending this to 
a probability measure. 
We write $\ol \om$ for elements of $\sD$.
Finally, we recall the definitions of the stopping times
$\tau_R$ and $\wt \tau_r$ from \eqref{e:tRdef} and \eqref{e:wtrdef} and the transition 
densities $p^\om_n(x,y)$ from \eqref{eq:hkdef}. To avoid difficulties 
due to $\sU$ being bipartite, we also define
\begin{equation}
 \wp^\om_n(x,y)= p^\om_n(x,y) + p^\om_{n+1}(x,y).
\end{equation}

\sms

Throughout this section, we will write $C(\lam)$ to denote expressions of the form $C \lam^p$ and $c(\lam)$ to denote expressions of the form $c \lam^{-p}$, where $c$, $C$ and $p$ are positive constants.

\sms

As in \cite{BJKS, KM} we define a (random) set $J(\lam)$:

\begin{definition} \label{jdef} 
\emph{ Let $\sU$ be the UST. 
For $\lambda \ge 1$ and $x \in \bZ^2$, let $J(x,\lambda)$ be the set of
those $R \in [1,\infty]$ such that the following all hold: \\
(1) $ |B_d(x,R)| \le  \lam g(R)^2  $,\\
(2) $\lam^{-1} g(R)^2  \le |B_d(x,R)| $,\\
(3)  $\Reff(x, B_d(x,R)^c) \ge \lam^{-1} R$. }
\end{definition}

\begin{proposition}\label{p:KMsat}
For $R\ge 1$, $\lam \ge 1$ and $x \in \bZ^2$, \\
(a) \begin{equation} \label{e:kmest}
 \bP( R \in J(x, \lam) ) \ge 1 - C e^{-c \lam^{1/9}};
 \end{equation}
(b) $$ \bE ( \Reff(0, B_d(0,R)^c) |B_d(0,R)|) \le C R g(R)^2. $$
Therefore conditions (1), (2) and (4) of \cite[Assumption 1.2]{KM} hold with $v(R) = g(R)^2$ and $r(R) = R$. 
\end{proposition}

\begin{proof}
(a) is immediate from Theorem \ref{t:main2} and Proposition \ref{p:kmest}(a), while (b) is exactly Proposition \ref{p:kmest}(b). We note that since $r(R)=R$, the condition
$\Reff(x,y) \le \lam r (d(x,y))$ in \cite[Definition 1.1]{KM} 
always holds for $\lam \ge 1$, so that our definition of $J(\lam)$ 
agrees with that in \cite{KM}. \qed
\end{proof}

We will see that the time taken by the random walk $X$ to move a distance $R$
is of order $R g(R)^2$. We therefore define
\begin{equation}
  F(R)= R g(R)^2,
\end{equation}
and let $f$ be the inverse of $F$. We will prove that the heat kernel $\wp_T(x,y)$ is of order $g(f(T))^{-2}$ and so we let
\begin{equation}
k(t) =  g(f(t))^2, \q t \ge 1.
\end{equation}
Note that we have $f(t) k(t) = f(t) g(f(t))^2 = F( f(t))=t$, so 
\begin{equation}
 \frac{1}{k(t)} = \frac{1}{g(f(t))^2} = \frac{f(t)}{t}.
\end{equation}
Furthermore, since $G(R) \approx R^{5/4}$, we have
\begin{align}
G(R) \approx R^{5/4}, \quad g(R) &\approx R^{4/5}, \quad F(R) \approx R^{13/5}, \\
f(R) \approx R^{5/13}, \quad  k(R) &\approx R^{8/13}, \quad R^2 G(R) \approx R^{13/4}.
\end{align}

We now state our results for the SRW $X$ on $\sU$, giving the asymptotic behaviour of $d(0,X_n)$, the transition densities
$\wp^\om_{n}(x,y)$, and the exit times $\tau_R$ and $\wt \tau_r$. 
We begin with three theorems which follow directly from Proposition \ref{p:KMsat} and \cite{KM}.
The first theorem gives tightness for some of these quantities, the second theorem gives expectations
with respect to $\bP$, and the third theorem gives `quenched' limits which hold $\bP$-a.s. 
In various ways these results make precise the intuition that
the time taken by $X$ to escape from a 
ball of radius $R$ is of order $F(R)$, that $X$ moves a distance of order $f(n)$ in time $n$,
and that the probability of $X$ returning to its initial point after $2n$
steps is the same order as $1/|B(0, f(n))|$, that is $g(f(n))^{-2} = k(n)^{-1}$.

\begin{theorem} \label{ptight}
Uniformly with respect to $n\ge 1$, $R\ge 1$ and $r \ge 1$,
\begin{align}
\label{pt-a}
 \bP \Big(\theta^{-1}\le \frac{ E^0_\omega \tau_R}{F(R) }  \le \theta \Big) &\to 1
\quad \text{ as } \theta\to \infty, \\
\label{pt-a2}
 \bP \Big(\theta^{-1}\le \frac{ E^0_\omega \wt \tau_r}{r^2 G(r) }  \le \theta \Big) &\to 1
\quad \text{ as } \theta\to \infty, \\
\label{pt-b}
  \bP (\theta^{-1}\le k(n) p_{2n}^\omega(0,0)\le \theta) &\to 1
\quad \text{ as } \theta\to \infty, \\
\label{pt-d}
 P^* \Big( \theta^{-1} < \frac{1+d(0,X_n)}{ f(n)}   < \theta  \Big) &\to 1
\quad \text{ as } \theta\to \infty.
\end{align}
\end{theorem}

\begin{theorem} \label{t:means}
There exist positive constants $c$ and $C$ such that for all $n\ge 1$, $R \ge 1$, $r\ge 1$,
\begin{align}
\label{e:mmean}
    cF(R) &\le \bE (E^0_\omega\tau_R) \le C F(R) , \\
\label{e:mEmean}
    c r^2 G(r) &\le \bE (E^0_\omega \wt \tau_r) \le C r^2 G(r)  , \\
\label{e:pmean}
 c k(n)^{-1}  &\le \bE (p_{2n}^\omega(0,0))\le C k(n)^{-1} , \\
\label{e:dmean}
       c f(n) &\le \bE (E_\omega^0 d(0,X_n)).
\end{align}
\end{theorem}

\begin{theorem}
\label{thm-rwre} There exist $\al_i < \infty$, and a subset $\Omega_0$
with $\bP(\Omega_0)=1$ such that the following statements hold.\\
(a) For each $\omega \in \Omega_0$ and $x \in \bZ^2$ there
exists $N_x(\omega)< \infty$ such that
\begin{align}  
\label{e:logpnlima}
   (\log\log n)^{-\al_1} k(n)^{-1} \le
 p^\omega_{2n}(x,x) \le (\log\log n)^{\al_1}  k(n)^{-1}, \quad  n\ge N_x(\omega).
\end{align}
In particular, $d_s(\sU) = 16/13$, $\bP$-a.s. \\
(b)  For each $\omega \in \Omega_0$ and $x \in \bZ^2$ there
exists $R_x(\omega)< \infty$ such that
\begin{align}
 \label{e:logtaulima}
   (\log\log R)^{-\al_2} F(R) &\le E^x_\omega \tau_R
 \le  (\log\log R)^{\al_2} F(R), \quad  R\ge R_x(\omega), \\
  \label{e:logtaulimE}
   (\log\log r)^{-\al_3} r^2 G(r)  &\le E^x_\omega \wt\tau_r
 \le  (\log\log r)^{\al_3} r^2 G(r)^2, \quad  r \ge R_x(\omega).
\end{align}
Hence
\begin{align}
 d_w(\sU)= \lim_{R \to \infty} \frac{\log  E^x_\omega \tau_R}{\log R} =  \frac{13}{5}, \q
 \q \lim_{r \to \infty} \frac{\log  E^x_\omega \wt\tau_r}{\log r} =  \frac{13}{4}.
\end{align}
(c)  Let $Y_n= \max_{0\le k \le n} d(0,X_k)$.
For each $\omega \in \Omega_0$ and $x \in \bZ^2$
there exist $\ol N_x(\overline \omega)$, $\ol R_x(\overline \omega)$
such that $P^x_\omega(\ol N_x <\infty)=P^x_\omega(\ol R_x <\infty)=1$,
and such that
\begin{align}
 \label{e:ynlim}
 (\log\log n)^{-\al_4} f(n) &\le Y_n(\overline \omega)
 \le (\log\log n)^{\al_4}  f(n),
 \quad n \geq \ol N_x(\overline \omega), \\
   \label{e:rnlim}
 (\log\log R)^{-\al_4} F(R) &\le \tau_R(\overline \omega)
 \le (\log\log R)^{\al_4}  F(R), \quad\quad R \ge \ol R_x(\overline \omega), \\
   \label{e:rnlimE}
 (\log\log r)^{-\al_4} r^2G(r)  &\le \wt \tau_r(\overline \omega)
 \le (\log\log r)^{\al_4}  r^2G(r), \q r \ge R_x(\overline \omega).
\end{align}
(d) Let $W_n = \{X_0,X_1,\ldots, X_n\}$ and let $|W_n|$ denote its cardinality.
For each $\omega \in \Omega_0$ and $x \in \bZ^2$,
\begin{equation}
\label{e:snlim}
  \lim_{n \to \infty} \frac{\log |W_n|}{\log n} = \frac{8}{13}, \quad
 P^x_\omega \text{-a.s.}.
\end{equation}
\end{theorem}

\sms The papers \cite{BJKS, KM} studied random graphs for which information
on ball volumes and resistances were only available from one point. 
These conditions were not strong enough to bound $E^0_\om d(0,X_n)$ or $\wp^\om_T(x,y)$ --
see \cite[Example 2.6]{BJKS}.
Since the UST is stationary, we have the same estimates available from every point $x$, 
and this means that stronger conclusions are possible.

\begin{thm} \label{t:dist}
There exist $N_0(\om)$ with $\bP( N_0 < \infty) =1$, $\alpha > 0$ and for 
all $q > 0$, $C_q$ such that 
 \begin{equation} \label{e:qdub}
  E^0_\om d(0, X_n)^q \le C_q f(n)^q  (\log n)^{\al q} \q \hbox{ for } n \ge N_0(\om).
 \end{equation}
Further, for all $n \geq 1$,
\begin{equation} \label{e:adub}
\bE(  E^0_\om d(0, X_n)^q)  \le C_q f(n)^q  (\log n)^{\al q}  . 
\end{equation}
\end{thm}

Write $\Phi(T,x,x) = 0$, and for $x \neq y$ let 
\begin{equation} \label{e:Phidef}
\Phi(T,x,y)= \frac{d(x,y)}{G((T/ d(x,y))^{1/2})}.
\end{equation}

\begin{theorem} \label{t:hk}
There exists a constant $\al>0$ and r.v. $N_x(\om)$ with
\begin{equation}\label{e:Nxtail}
 \bP( N_x \ge n ) \le C e^{ -c (\log n)^2 }
\end{equation}
such that provided $ F(T) \vee |x-y| \ge N_x(\om)$ and $T \ge d(x,y)$, then
writing $A= A(x,y,T) =C  (\log( |x-y| \vee F(T)))^\alpha$,
\begin{align}\label{e:hkb2}
\frac{1} {A k(T)} \exp\Big( - A \Phi(T,x,y)   \Big) \le
 \wp_T(x,y) &\le \frac{A}{k(T)} \exp\Big( -  A^{-1} \Phi(T,x,y)   \Big).
\end{align}
\end{theorem}

\begin{remark} {\rm 
If we had $G(n) \asymp n^{5/4}$ then since $d_f = 8/5$ and $d_w=1+d_f$,
we would have
\begin{equation}
 \Phi(T,x,y) \asymp \Big( \frac{d(x,y)^{d_w}}{T} \Big)^{1/(d_w-1)},
\end{equation}
so that, except for the logarithmic term $A$, the bounds in 
\eqref{e:hkb2} would be of the same form as those obtained in the diffusions
on fractals literature. 
} \end{remark}

\sms Before we prove Theorems \ref{ptight} -- \ref{t:hk}, we summarize some properties of
the exit times $\tau_R$.

\begin{proposition}\label{p:kmtau}
Let $\lam \ge 1$ and $x \in \bZ^2$. \\
(a) If $R, R/(4 \lam) \in J(x,\lam)$ then
\begin{equation}\label{e:fixtau}
 c_1(\lam) F(R)  \le E^x_\om \tau(x,R)  \le C_2(\lam) F(R). 
  \end{equation}
(b) Let $0< \eps \le c_3(\lam)$. Suppose that $R, \eps R, c_4(\lam)\eps R \in J(x, \lam)$.
Then
\begin{equation}\label{e:epstau}
   P^x_\om( \tau(x,R) < c_5(\lam) F(\eps R) ) \le C_6(\lam) \eps. 
\end{equation}
\end{proposition}

\proof
This follows directly from \cite[Proposition 2.1]{BJKS} and \cite[Proposition 3.2, 3.5]{KM}.
 \qed

\sm {\bf Proof of Theorems \ref{ptight}, \ref{t:means},
and \ref{thm-rwre}. }
All these statements, except those relating to $\wt\tau_r$,
follow immediately from Proposition \ref{p:KMsat} and 
Propositions 1.3 and 1.4 and Theorem 1.5 of \cite{KM}.
Thus it remains to prove 
\eqref{pt-a2}, \eqref{e:mEmean}, \eqref{e:logtaulimE} and \eqref{e:rnlimE}.
By the stationarity of $\sU$ it is enough to consider the case $x=0$.

Recall that $U_r$ denotes the connected component of $0$ in $\sU \cap B(0,r)$, and therefore
$$ \wt \tau_r = \min\{ n \ge 0: X_n \not\in U_r\}. $$
Let 
$$ H_1(r, \lam) = \{  B_d(0, \lam^{-1} G(r)) \subset U_r \subset
 B_d(0, \lam G(r) \}. $$
On $H_1(r, \lam)$ we have
\begin{equation} \label{e:tauineq}
 \tau_{\lam^{-1} G(r) } \le \wt \tau_r \le \tau_{\lam G(r) }, 
 \end{equation}
while by Theorem \ref{t:Vub} and Proposition \ref{p:Ur} we have
for $r \ge 1$, $\lam \ge 1$,
$$ \bP( H_1(r, \lam) ^c ) \le e^{ - c \lam^{2/3} }. $$
The upper bound in \eqref{pt-a2} will follow from \eqref{e:mEmean}. For the lower bound,
on $H_1(r,\lam)$ we have, writing $R = \lam^{-1} G(r)$,
\begin{align}
 \frac{ E^0_\om \wt \tau_r }{r^2 G(r) } \ge  \frac{ E^0_\om \tau_{R} }{ F(R) } \cdot  \frac{ F(R)}{r^2 G(r) },
 \end{align}
while $F(R)/ r^2 G(r) \ge \lam^{-3}$ by Lemma \ref{l:vgrwth}.
So
\begin{align}
\bP\Big(  \frac{ E^0_\om \wt \tau_r }{r^2 G(r) } < \lam^{-4} \Big)
\le \bP ( H_1(r,\lam)^c) + \bP\Big(   \frac{ E^0_\om \tau_{R} }{ F(R) } < \lam^{-1}\Big ),
\end{align}
and the bound on the lower tail in \eqref{pt-a2} follows from \eqref{pt-a}.

\sms We now prove the remaining statements in Theorem \ref{thm-rwre}.
Let $r_k=e^k$, and $\lam_k = a( \log k)^{3/2}$, and choose $a$ large 
enough so that
$$ \sum_k  \exp( - c \lam_k^{2/3}) < \infty. $$
Hence by Borel-Cantelli there exists a r.v. $K(\om)$ with
$\bP(K< \infty)=1$ such that
$H_1(r_k, \lam_k)$ holds for all $k \ge K$. 
So if $k$ is sufficiently large, and $\alpha_2$ is as in \eqref{e:logtaulima},
\begin{align*}
  E^0_\om \wt \tau_{r_k} \le E^0_\om \tau_{\lam_k G(r_k)}
 &\le [\log \log (\lam_k G(r_k))]^{\al_2} \lam_k G(r_k) g( \lam_k G(r_k))^2 \\
&\le C (\log k)^{\al_3} r_k^2  G(r_k) \\
&= C (\log \log r_k)^{\al_3} r_k^2 G(r_k).
\end{align*}
Since $\wt \tau_r$ is monotone in $r$,
the upper bound in \eqref{e:logtaulimE} follows.
A very similar argument gives the lower bound, and also \eqref{e:rnlimE}.

It remains to prove \eqref{e:mEmean}. 
A general result on random walks (see e.g. \cite{BJKS}, (2.21))
implies that
$$ E^0_\om \wt \tau_r \le \Reff(0, U_r^c) \sum_{x \in U_r} \mu_x \leq C r^2 \Reff(0, U_r^c) . $$
Let $z$ be the first point on the path $\gam(0,\infty)$ outside
$B(0,r)$. Then $\Reff(0, U_r^c) \le d(0, z)$, and since
 $\gam(0,\infty)$ has the law of an infinite LERW, 
$\bE d(0,z) \le \bE \wh M_{r+1} \le C G(r)$. Hence
$$ \bE  (E^0_\om \wt \tau_r)  \le C r^2 G(r). $$

For the lower bound, let
$$ H_2(r, \lam) =\{   \lam^{-1}G(r), (2 \lam)^{-2}G(r) \in J(\lam) \}. $$
Choose $\lam_0$ large enough so that $\bP( H_1(\lam_0, r) \cap H_2(\lam_0, r)) \ge \fract12$.
If $H_2(r, \lam_0)$ holds then by Proposition \ref{p:kmtau}, writing 
$R = \lam_0^{-1} G(r)$,
$$ E^0_\om \tau_R  \ge c(\lam_0) R g(R)^2. $$
So, since $R g(R)^2 \ge c(\lam_0) r^2 G(r)$,
\begin{align*}
 \bE E^0_\om \wt \tau_r &\ge  \bE( E^0_\om \wt \tau_r; H_1(\lam_0, r) \cap H_2(\lam_0, r) )\\
 &\ge  \bE( E^0_\om \tau_{R} ; H_1(\lam_0, r) \cap H_2(\lam_0, r) )\\
 &\ge \fract12 c(\lam_0) R g(R)^2 \\
 &\ge c(\lam_0) r^2 G(r). 
 \end{align*}
 \qed

\sms We now turn to the proofs of Theorems \ref{t:dist} and \ref{t:hk}, and 
begin with a slight simplification of Lemma 1.1 of \cite{BBSC}.

\begin{lemma}\label{lem:bbi}
There exists $c_0 > 0$ such that the following holds. Suppose  we have nonnegative r.v. $\xi_i$ which satisfy, for some $t_0>0$,
$$ \bP( \xi_i \le t_0 | \xi_1, \dots, \xi_{i-1} ) \le \fract12. $$
Then 
\begin{equation}
 \bP( \sum_{i=1}^n \xi_i < T ) \le  \exp( - c_0 n + T/t_0  ). 
\end{equation}
\end{lemma}

\proof Write $\sF_i = \sigma(\xi_1, \dots \xi_i)$.
Let $\th = 1/t_0$, and let $e^{-c_0} = \fract12 (1 + e^{-1})$. Then
\begin{align*}
 \bE (e^{-\th \xi_i} | \sF_{i-1}) &\le \bP( \xi_i < t_0 | \sF_{i-1} ) + e^{-\th t_0} \bP( \xi_i > t_0  | \sF_{i-1}) \\
  &= \bP( \xi_i < t_0 | \sF_{i-1} ) ( 1- e^{-\th t_0} ) +  e^{-\th t_0} \\
  &\le \fract12 ( 1 + e^{-\th t_0} ) = e^{-c_0}.
\end{align*}
Then
\begin{align*}
 \bP(  \sum_{i=1}^n \xi_i < T ) =  \bP(  e^{-\th \sum_{i=1}^n \xi_i} > e^{-\th T} ) 
 \le e^{\th T} \bE(  e^{-\th \sum_{i=1}^n \xi_i} )
\le e^{\th T}  e^{-n c_0}. 
\end{align*}
\qed 

We also require the following lemma which is an immediate consequence
of the definitions of the functions $F$ and $G$.

\begin{lemma}\label{lem:mc}
Let $R\ge 1$, $T \ge 1$, and 
\begin{equation}\label{e:msat5}
  b_0= \frac{R}{ G((T/R)^{1/2})}.
\end{equation}
Then,
\begin{align} \label{e:m41}
 {R}/{b_0} &= G( (T/R)^{1/2}) = f(T/b_0), \\
 \label{e:m42}
b \le b_0 &\Leftrightarrow T/b \le F(R/b)  
\Leftrightarrow f(T/b) \le R/b. 
\end{align}
Also, if $\th<1$ and $\th R \ge 1$, then
\begin{align} \label{e:Fpower}
 c_7 \th^3 F(R) \le F(\th R) \le C_8 \th^2 F(R), \\
 \label{e:fpower}
  c_7 \th^{1/2} f(R) \le f(\th R) \le C_8 \th^{1/3} f(R).
 \end{align}
\end{lemma}

\sms

For $x \in \bZ^2$, let 
\begin{align*}
 A_x(\lam, n) =\{\om:  R' \in J(y, \lam) \, \hbox { for all } y \in B(x, n^2), 
1 \le R' \le n^2 \}.
\end{align*}
and let $A(\lam, n) = A_0(\lam,n)$.

\begin{proposition} \label{p:tautail}
Let $\lam \geq 1$ and suppose that $1 \leq R \le n$, 
\begin{equation} \label{e:Tcond}
T \ge  C_{9}(\lam) R, % RM:redundant? \ge C_6(\lam)
\end{equation} 
and $A(\lam, n)$ occurs.
Then,
\begin{equation}\label{e:tautail2}
 P^0_\om( \tau_R < T ) \le  C_{10}(\lam)  \exp\left(- c_{11}(\lam) \frac{R}{G( (T/R)^{1/2}) }  \right). 
\end{equation}
\end{proposition}

\proof 
In this proof, the constants $c_i(\lam)$, $C_i(\lam)$ for $1\le i \le 8$ 
will be as in Proposition \ref{p:kmtau} and Lemma \ref{lem:mc}, 
and $c_0$ will be as in Lemma \ref{lem:bbi}.
We work with the probability $P^0_\om$, so that $X_0=0$.

Let $b_0= R/ G((T/R)^{1/2})$ be as in \eqref{e:msat5}, and define the quantities
\begin{eqnarray*}
\eps = (2C_6(\lam))^{-1}, \qquad & \qquad \th = \frac14 C_8^{-1} c_0 c_5(\lam) \eps^2, 
 \qquad & \qquad C^*(\lam) = 2\th^{-1},  \\
 m = \lfloor \th b_0 \rfloor,  \qquad & \qquad R' = R/m,  
\qquad & \qquad  t_0 = c_5(\lam) F(\eps R').
\end{eqnarray*}

We now establish the key facts that we will need about the quantities defined above. 
We can assume that $b_0 \ge C^*(\lam)$ for if $b_0 \le C^*(\lam)$, then by adjusting the 
constants $C_{10}(\lam)$ and $c_{11}(\lam)$ we will still obtain \eqref{e:tautail2}. Therefore, 
\begin{equation} \label{e:msat2}
1 \leq \fract12 \th b_0 \leq  m \le \th b_0.
\end{equation}
Furthermore, 
since $m / \th \leq b_0$, $\th R /m = G((T/R)^{1/2}) \geq 1$ and $\theta/\eps < 1$, 
we have by Lemma \ref{lem:mc} that
$$  T/m \le \th^{-1} F(\th R/m) \le C_8 \th \eps^{-2} F( \eps R/m)
\le  \fract14 c_0 c_5(\lam)  F( \eps R/m) = \frac{1}{4} c_0 t_0. $$
Therefore,
\begin{equation} \label{e:msat3}
T/t_0 < \half c_0 m.
\end{equation}
Finally, we choose 
$$ C_{9}(\lam) \geq g(c_4(\lam)^{-1} \eps^{-1} \th)^2,$$
so that if $T/R \geq C_{9}(\lam)$,
then
$$ G((T/R)^{1/2}) \geq c_4(\lam)^{-1} \eps^{-1} \th,$$
and therefore
\begin{equation} \label{e:msat1}
c_4(\lam) \eps R' \geq c_4(\lam) \eps R \theta^{-1} b_0^{-1} \geq 1.
\end{equation}

Having established \eqref{e:msat2}, \eqref{e:msat3} and \eqref{e:msat1}, the proof of the 
Proposition is straightforward. Let $\sF_n = \sigma(X_0, \dots, X_n)$.
Define stopping times for $X$ by 
\begin{align*}
 T_0 &= 0, \\
 T_k &= \min \{ j \ge T_{k-1}: X_j \not\in B_d( X_{T_{k-1}}, R'-1) \},
\end{align*}
and let $\xi_k = T_k - T_{k-1}$. Note that $T_m \leq \tau_R$, and that if $k \le m$, then 
$$ X_{T_k} \in B_d(0, kR' ) \subset B_d(0,n) \subset B(0,n).$$
Therefore, since \eqref{e:msat1} holds and $A(\lam, n)$ occurs, 
we can apply Proposition \ref{p:kmtau} to obtain that  
\begin{equation*}
 P^0_\om( \xi_k < c_5(\lam) F(\eps R' )| \sF_{k-1} ) \le C_6(\lam) \eps = \fract12.  
\end{equation*}
Hence by Lemma \ref{lem:bbi} and \eqref{e:msat3},
\begin{align*}%\label{e:tbnd2}
 P^0_\om( \tau_R < T ) &\le  P^0_\om( \sum_1^m \xi_i < T ) \\
 &\le \exp( -c_0 m + T/t_0) \\
&\le \exp(-c_0/2 m) \\
&\le \exp\left(- c_{11}(\lam) \frac{R}{G( (T/R)^{1/2}) }  \right).  
\end{align*}

\qed

\sm {\bf Proof of Theorem \ref{t:dist} } 
We will prove Theorem \ref{t:dist} with $T$ replacing $n$. Let $R = f(T)$; 
we can assume that $T$ is large enough so that $R \ge 2$.  
We also let $C_9(\lam)$, $C_{10}(\lam)$ and $c_{11}(\lam)$ be as in 
Proposition \ref{p:tautail}, and let $p > 0$ be such that 
$C_{i}(\lam) \leq C \lam^{p}$, $i=9,10$ and $c_{11}(\lam) \geq c \lam^{-p}$.

We have
\begin{align} \nn
 E^0_\om d(0, X_T)^q &\le R^q 
  + E^0_\om\Big(  \sum_{k=1}^\infty 1_{( e^{k-1} R \le d(0, X_T) < e^k R)} d(0, X_T)^q \Big)\\
  \label{e:d-est}
  &\le R^q + R^q  \sum_{k=1}^\infty e^{kq} P^0_\om(  e^{k-1} R \le  d(0, X_T) \le e^{k} R ) . 
\end{align}

By \eqref{e:kmest} we have
\begin{equation} \label{e:apbnd}
 \bP( A(\lam,n)^c ) \le 4 n^3 e^{-c\lam^{1/9} } \le \exp( -c \lam^{1/9} + C \log n ). 
\end{equation}
Let $\lam_k = k^{10}$. Then 
$\sum_k \bP( A(\lam_k, e^k)^c) < \infty$, and so by Borel-Cantelli there exists
$K_0(\om)$ such that $A(\lam_k, e^k)$ holds for all $k \ge K_0$.
Furthermore, we have 
\begin{equation} \label{e:Ktail}
 \bP( K_0 \ge n ) \le  C e^{-c n^{10/9}}.
\end{equation}

Suppose now that $k \geq K_0$. To bound the sum \eqref{e:d-est}, we consider two ranges of $k$. 
If $C_{9}(\lam_k)e^{k-1} R > T$, then we let $A_k = B_d(0, e^k R) - B_d(0, e^{k-1} R)$, 
and by the Carne-Varopoulos bound (see \cite{Ca}),
\begin{align} \nn
 e^{kq} P^0_\om(  e^{k-1} R \le  d(0, X_T) \le e^{k} R ) 
 &\le   e^{kq} \sum_{y \in A_k } P^0_\om(X_T =y) \\ \nn
 &\le   e^{kq} \sum_{y \in A_k } C \exp( - d(0,y)^2/2T )\\ \nn
 &\le C e^{kq} (e^k R)^2 \exp( - (e^{k-1} R)^2/2T ) \\ \nn
 &\le C \exp( - C_{9}(\lam_k)^{-1} e^kR + 2 \log(e^k R) + kq  )\\ \label{e:dsum1}
 &\le C \exp( - c k^{-10p} e^k  + C_q k  ).
 \end{align}

On the other hand, if $C_{9}(\lam_k)e^{k-1} R \leq T$, then we
let $m= \lceil k + \log R \rceil$, so that $e^k R \le e^m < e^{k+1} R$. 
Then by Proposition  \ref{p:tautail},
\begin{align} \nn
e^{kq}  P^0_\om(  e^{k-1} R \le  d(0, X_T) \le e^k R ) 
 &\le e^{kq} P^0_\om ( \tau_{e^{k-1} R} < T) \\ \nn
 &\le e^{kq}  C_{10}(\lam_m) \exp\left( - c_{11}(\lam_m) 
        \frac{e^{k-1}R}{G((e^{-k+1}T/R)^{1/2})} \right) \\ \nn
&\le e^{kq}  C m^{10p} \exp\left( - c m^{-10p} e^{k} \frac{R}{G((T/R)^{1/2})} \right) \\
 &\le C  (k + \log R)^{10p} \exp( - c (k + \log R)^{-10p} e^k + kq ). 
 \end{align}

 Let $k_1= 20p \log \log R$. Then if $k \ge k_1$,
$$ (k + \log R)^{10p} \le ( k +  e^{k/(20p)})^{10p} \le C e^{k/2}. $$
 Hence for    $k \ge k_1$,
\begin{align} \label{e:dsum2}
e^{kq}  P^0_\om(  e^{k-1} R \le  d(0, X_T) \le e^k R ) 
 &\le C \exp( - c e^{k/2} + C_q k ). 
 \end{align} 
Let $K' = K_0 \vee k_1$. Then
since the series given by \eqref{e:dsum1} and \eqref{e:dsum2} both converge, 
\begin{align}\nn 
 \sum_{k=1}^\infty e^{kq} P^0_\om(  e^{k-1} R \le  d(0, X_T) \le e^k R ) 
&\le  \sum_{k=1}^{K'-1} e^{kq}  +C_q  \\ \nn
&\le e^{K' q} + C_q \\ \nn
&\le e^{K_0 q } + (\log R)^{20 p q}  + C_q. 
\end{align}
Hence since $R \le T$, we have that for all $T \geq N_0 = e^{e^{K_0}}$ 
\begin{equation}\label{e:dub1}
  E^0_\om d(0, X_T)^q \le C_q R^q( (\log T)^q + (\log T )^{20 pq} ), 
\end{equation}
so that \eqref{e:qdub} holds. Taking expectations in \eqref{e:dub1}
and using \eqref{e:Ktail} gives \eqref{e:adub}.
\qed

\begin{remark}{\rm
It is natural to ask if \eqref{e:adub} holds without the term in $\log T$, as with the
averaged estimates in Theorem \ref{t:means}.
It seems likely that this is the case; such an averaged estimate was proved for
the incipient infinite cluster on regular trees in \cite[Theorem 1.4(a)]{BK06}. 
The key to obtaining such a bound is to control the exit times 
$\tau_{e^k R}$; this was done above using the events
$A(\lam, n)$, but this approach is far from optimal. 
The argument of Proposition \ref{p:tautail} goes through if 
only a positive proportion of the points $X_{T_k}$ are at places
where the estimate \eqref{e:epstau} can be applied.
This idea was used in \cite{BK06} -- see the definition of the event $G_2(N,R)$
on page 48.

Suppose we say that $B_d(x,R)$ is $\lam$-bad if $R \not\in J(x, \lam)$. Then
it is natural to conjecture that there exists $\lam_c$ such that for $\lam > \lam_c$ 
the bad balls fail to percolate on $\sU$.
Given such a result (and suitable control on the size of the
clusters of bad balls) it seems plausible that the methods of this paper and \cite{BK06} would then
lead to a bound of the form 
$\bE( E^0_\om d(0,X_T)^q) \le C_q f(T)^q$.
}
\end{remark}

\ms

We now use the arguments in \cite{BCK}
to obtain full heat kernel bounds for $p_T(x,y)$ and thereby prove Theorem \ref{t:hk}. Since the techniques 
are fairly standard, we only give full details for  the less familiar steps.

\begin{lemma} \label{lem:alam}
Suppose $A(\lam, n)$ holds.  Let
$x, y \in B(0,n)$. Then \\
(a)  
\begin{equation}\label{e:ptondub}
p_T(x,y) \le C_{12}(\lam) k(T)^{-1}, \q \hbox{ if } 1 \le T \le F(n).
\end{equation}
(b) 
\begin{equation}\label{e:ptndlb}
\wp_T(x,y) \ge c_{13}(\lam) k(T)^{-1}, \q \hbox{ if } 1 \le T \le F(n) \hbox{ and }  d(x,y) \le c_{14}(\lam) f(T).
\end{equation}
\end{lemma}

\proof 
If $x=y$ then (a) is immediate from \cite[Proposition 3.1]{KM}. Since
$p_T(x,y)^2 \le \wp_T(x,x) \wp_T(y,y)$, the general case then follows. \\
(b) The bound when $x=y$ is given by \cite[Proposition 3.3(2)]{KM}. We also
have, by \cite[Proposition 3.1]{KM},
\begin{equation*}
 |\wp_T(x,y) - \wp_T(x,z)|^2 \le \frac{c}{T} d(y,z) p_{2 \lfloor T/2 \rfloor}(x,x).
\end{equation*}
Therefore using (a),
\begin{align*}
 \wp_T(x,y) &\ge \wp_T(x,x) -  |\wp_T(x,x) - \wp_T(x,y)|  \\
 &\ge c(\lam) k(T)^{-1} - \Big( C(\lam) d(x,y) T^{-1} k(T)^{-1} \Big)^{1/2} \\
&= c(\lam) k(T)^{-1} \Big( 1-  \big( C(\lam) d(x,y) T^{-1} k(T) \big)^{1/2} \Big).
\end{align*}
Since $k(T)/T = f(T)^{-1}$, \eqref{e:ptndlb} follows. \qed

Recall that $\Phi(T,x,x) = 0$, and for $x \neq y$, 
$$ \Phi(T,x,y)= \frac{d(x,y)}{G((T/ d(x,y))^{1/2})}.$$

\begin{proposition} \label{p:ubfix}
Suppose that $A(\lam, n)$ holds. Let $x, y \in B(0,n)$. 
If $ d(x,y) \le T \le F(n)$, then
\begin{align}\label{e:hkb}
\frac{c(\lam)}{k(T)} \exp\Big( - C(\lam)  \Phi(T,x,y)   \Big) \le
 \wp_T(x,y) &\le \frac{C(\lam)}{k(T)} \exp\Big( - c(\lam)  \Phi(T,x,y)   \Big). 
\end{align}
\end{proposition}

\proof
Let $R = d(x,y)$. In this proof we take $c_{13}(\lam)$ and $c_{14}(\lam)$ to be as in \eqref{e:ptndlb}.

We will choose a constant $C^*(\lam)\ge 2$ later.
Suppose first that $R  \le T \le C^*(\lam)R$. Then the upper bound in \eqref{e:hkb}
is  immediate from the Carne-Varopoulos bound. If $R + T$ is even and
then we have $p_T(x,y) \ge 4^{-T}$, and this gives the lower bound.

We can therefore assume that $T \ge C^*(\lam)R$.
The upper bound follows from the bounds \eqref{e:ptondub} and
\eqref{e:tautail2} by the same argument as in \cite[Proposition 3.8]{BCK}. 

It remains to prove the lower bound in the case when $T \ge C^*(\lam)R$,
and for this we use a standard chaining technique which
derives \eqref{e:hkb} from the `near diagonal lower bound' \eqref{e:ptndlb}. 
For its use in a discrete setting see for example \cite[Section 3.3]{BCK}.
As in Lemma \ref{lem:mc}, we set
\begin{equation}\label{e:mchoose}
b_0 =   \frac{R}{G((T/R)^{1/2})}.
\end{equation}
If $b_0< 1$ then we have from Lemma \ref{lem:mc}  that
$R \le C_8 b_0^{2/3} f(T)$. 
If $C_8 b_0^{2/3} \le  c_{14}(\lam)$ then  
$R \le c_{14}(\lam) f(T)$ and the lower bound in \eqref{e:hkb} follows
from \eqref{e:ptndlb}. We can therefore assume that  
$C_8 b_0^{2/3} >   c_{14}(\lam)$. 
We will choose $\th>  2 (c_{14}/C_8)^{-3/2}$ later; this then implies that $\th b_0 \ge 2$.
%MB end of changes
Let $m = \lfloor \th b_0 \rfloor$;  we have $\half \th b_0 \le m \le \th b_0$. 
Let $r = R/m$, $t=T/m$; we will require that both $r$ and $t$ are greater than $4$. 
Choose integers $t_1, \dots, t_m$ so that
$|t_i - t| \le 2$ and $\sum t_i = T$. 
Choose a chain $x=z_0, z_1, \dots, z_m = y$ of points so that
$d(z_{i-1}, z_i) \le 2r$, and let $B_i = B(z_i, r)$. If $x_i \in B_i$ for $1\le i \le m$ then
$d(x_{i-1}, x_i) \le 4r$. We choose $\th$ so that we have
\begin{equation} \label{e:usend}
  \wp_{t_i}(x_{i-1},x_i) \ge c_{13}(\lam) k(t)^{-1} 
\hbox { whenever } x_{i-1} \in B_{i-1}, \, x_i \in B_i. 
\end{equation}
By \eqref{e:ptndlb} it is sufficient for this that 
\begin{equation} \label{e:rtcond}
 4R/m=  4r \le c_{14}(\lam) f(t/2) = c_{14}(\lam) f(T/2m).
\end{equation}
Since $2m/\th \ge b_0$, Lemma \ref{lem:mc} implies that
$f( \th T/(2m)) \ge \th R/(2m)$, and therefore
\begin{equation} \label{e:et4}
 4R/m \le 8 \th^{-1} f(\th T/(2m)) \le C \th^{-1/3} f(T/2m),
\end{equation}
and so taking $\th =\max(2 (c_{14}/C_8)^{-3/2}, (C/c_3(\lam))^3)$ gives \eqref{e:rtcond}.
The condition $T \ge C^*(\lam)R$ implies that 
$f(T/b_0) =  R/b_0 \ge G(C^*(\lam))$, so taking $C^*$ large enough ensures
that both $r$ and $t$ are greater than 4.

The Chapman-Kolmogorov equations give
\begin{align}  \nn 
 \wp_T(x,y) &\ge  \sum_{x_1 \in B_1} \dots  \sum_{x_{m-1} \in B_{m-1}}
  p_{t_1}(x_0,x_1)\mu_{x_1} p_{t_2}(x_1,x_2)\mu_{x_2} \dots \\ \label{e:cklb} 
   & \qq \qq p_{t_{m-1}}(x_{m-2}, x_{m-1})\mu_{x_{m-1}} \wp_{t_{m}}(x_{m-1},y).
 \end{align}
Since $x_{m-1} \in B_{m-1} $ we have
$\wp_{t_{m}}(x_{m-1},y) \ge c_{13}(\lam) k(t)^{-1} \ge c_{13}(\lam) k(T)^{-1}$.
Note that exactly one of $p_t(x,y)$ and $p_{t+1}(x,y)$ can be non-zero.
Using this, and \eqref{e:usend} we deduce that for $1\le i \le m-1$,
\begin{equation}
   \sum_{x_{i} \in B_{i}} p_{t_i}(x_{i-1}, x_i) \mu_{x_i} 
    \ge c(\lam) k(t)^{-1} g(r)^2. 
\end{equation}
The choice of $m$ implies that $c'(\lam) f(t) \le r \le c(\lam) f(t) $, and therefore
$$ k(t)^{-1} g(r)^2 = g(r)^2/g(f(t))^2 \ge c(\lam). $$
So we obtain
\begin{equation}
 \wp_T(x,y) \ge k(T) c(\lam)^m \ge k(T) \exp( - c(\lam) R / G((T/R)^{1/2} )).
\end{equation}
\qed

\sm {\bf Proof of Theorem \ref{t:hk} } 
As in the proof of Theorem \ref{t:dist}, we have that by
by \eqref{e:kmest}
\begin{equation*} 
 \bP( A(\lam,n)^c ) \le 4 n^3 e^{-c\lam^{1/9} } \le \exp( -c \lam^{1/9} + C \log n ). 
\end{equation*}
Therefore if we let $\lam_n = (\log n)^{18}$, then by Borel
Cantelli, for each $x \in \bZ^2$ there exists $N_x$ such that
$A_x(\lam_n,n)$ holds for all $n \ge N_x$. Further we have
that 
$$  \bP( N_x \ge n ) \le C e^{ -c (\log n)^2 }.$$ 

Let $x, y \in \bZ^2$ and $T \ge 1$. To apply the bound in Proposition \ref{p:ubfix}
we need to find $n$ such that 
$T \le F(n)$, $y \in B(x,n)$ and $n \ge N_x$. 
Hence if $F(T) \vee |x-y| \ge N_x$ we can take $n = F(T) \vee |x-y|$, to obtain
\eqref{e:hkb} with constants $c(\lam_n) = c (\log n)^p$. 
Choosing $\al$ suitably then gives \eqref{e:hkb2}. \qed

\ms
\begin{remark} {\rm
If both $d(x,y)=R$ and $T$ are large then since $d_w = 13/5$
$$ \Phi(x,y) \simeq   R ((T/R)^{1/2})^{-5/4} = \frac{R^{13/8}}{T^{5/8}} 
= (R^{d_w}/T)^{1/(d_w-1)}.$$
Thus the term in the exponent
takes the usual form one expects for heat kernel bounds on a regular graph with fractal growth
-- see the conditions UHK$(\beta)$ and  LHK$(\beta)$ on page 1644 of \cite{BCK}.
}
\end{remark}

\sm {\bf Acknowledgment} The first author would like to thank Adam
Timar for some valuable discussions on stationary trees in
$\bZ^d$. The second author would like to thank Greg Lawler for help in
proving Lemma \ref{condesc}.

\end{document}